\documentclass[a4paper,oneside,10pt]{article}%
\usepackage{amsmath}
\usepackage{amsfonts}
\usepackage{amssymb}
\usepackage{color}
\usepackage{graphicx}
\usepackage{bbm}
\usepackage{natbib}
\usepackage{hyperref}%
\providecommand{\U}[1]{\protect \rule{.1in}{.1in}}

\newtheorem{theorem}{Theorem}[section]

\newtheorem{corollary}[theorem]{Corollary}

\newtheorem{definition}[theorem]{Definition}

\newtheorem{lemma}[theorem]{Lemma}

\newtheorem{proposition}[theorem]{Proposition}
\newtheorem{remark}[theorem]{Remark}

\newenvironment{proof}[1][Proof]{\noindent \textbf{#1.} }{\  \rule{0.5em}{0.5em}}
\numberwithin{equation}{section}
\begin{document}

\title{A robust $\alpha$-stable central limit theorem under sublinear expectation
without integrability condition}
\author{Lianzi Jiang\thanks{College of Mathematics and Systems Science, Shandong
University of Science and Technology, Qingdao, Shandong 266590, China.
jianglianzi95@163.com.}
\and Gechun Liang\thanks{Department of Statistics, The University of Warwick,
Coventry CV4 7AL, U.K. g.liang@warwick.ac.uk.}}
\date{}
\maketitle

\textbf{Abstract}. This article relaxes the integrability condition imposed in
the literature for the robust $\alpha$-stable central limit theorem under
sublinear expectation. Specifically, for $\alpha \in(0,1]$, we prove that the
normalized sums of i.i.d. non-integrable random variables $\big
\{n^{-\frac{1}{\alpha}}\sum_{i=1}^{n}Z_{i}\big \}_{n=1}^{\infty}$ converge in
distribution to $\tilde{\zeta}_{1}$, where $(\tilde{\zeta}_{t})_{t\in
\lbrack0,1]}$ is a multidimensional nonlinear symmetric $\alpha$-stable
process with a jump uncertainty set $\mathcal{L}$. The limiting $\alpha
$-stable process is further characterized by a fully nonlinear partial
integro-differential equation (PIDE)
\[
\left \{
\begin{array}
[c]{l}%
\displaystyle \partial_{t}u(t,x)-\sup \limits_{F_{\mu}\in \mathcal{L}}\left \{
\int_{\mathbb{R}^{d}}\delta_{\lambda}^{\alpha}u(t,x)F_{\mu}(d\lambda)\right \}
=0,\\
\displaystyle u(0,x)=\phi(x),\  \  \  \forall(t,x)\in \lbrack0,1]\times
\mathbb{R}^{d},
\end{array}
\right.
\]
where
\[
\delta_{\lambda}^{\alpha}u(t,x):=\left \{
\begin{array}
[c]{ll}%
u(t,x+\lambda)-u(t,x)-\langle D_{x}u(t,x),\lambda \mathbbm{1}_{\{|\lambda
|\leq1\}}\rangle, & \alpha=1,\\
u(t,x+\lambda)-u(t,x), & \alpha \in(0,1).
\end{array}
\right.
\]
The main tools are a weak convergence approach to obtain the limiting process, a L\'evy-Khintchine representation of the nonlinear $\alpha$-stable process and a truncation technique to estimate the corresponding $\alpha$-stable
L\'{e}vy measures. As a byproduct, the article also provides a probabilistic
approach to prove the existence of the above fully nonlinear PIDE. \newline


\textbf{MSC-classification}. 60F05, 60G51, 60G52, 60G65, 45K05

\section{Introduction}

Non-Gaussian $\alpha$-stable distributions allow for large fluctuations and
are thus ideal for modeling high variability and heavy tail events. They are
the only non-Gaussian limiting distributions as attractors of normalized sums
of i.i.d. random variables $\big \{n^{-\frac{1}{\alpha}}\sum_{i=1}^{n}%
Z_{i}\big
\}_{n=1}^{\infty}$, known as $\alpha$-stable central limit theorems. Recently,
such type of central limit theorems have been generalized to a robust case
under model uncertainty (see \cite{BM2016}, \cite{HJL2021} and \cite{HJLP2022}%
) for $\alpha \in(1,2)$. The case $\alpha \in(0,1]$ is left open
because $Z_{i}$ is not integrable in such a situation. In this
paper, we study this more challenging case by establishing a new
robust $\alpha$-stable central limit theorem without integrability
condition.

The theory of robust probability and expectation has been developed by
\cite{P2004,P2007,P20081,P2010}, who introduced the notion of sublinear
expectation space, called $G$-expectation space. Peng used it to evaluate
random outcomes over a family of possibly mutually singular probability
measures instead of a single probability measure. A seminal result in this
theory is Peng's robust central limit theorem established in
\cite{P20082,P2019}. He showed that, under certain moment conditions,
the i.i.d. sequence $\{(X_{i},Y_{i})\}_{i=1}^{\infty}$ on a sublinear
expectation space $(\Omega,\mathcal{H},\mathbb{\hat{E}})$ converges in law to
a $G$-distributed random variable $(\xi,\eta)$, i.e.
\[
\lim_{n\rightarrow \infty}\mathbb{\hat{E}}\bigg[\phi \bigg(\frac{1}{\sqrt{n}%
}\sum_{i=1}^{n}{X_{i}},\frac{1}{n}\sum_{i=1}^{n}Y_{i}%
\bigg)\bigg]=\mathbb{\tilde{E}}[\phi(\xi,\eta)],
\]
for any test function $\phi$. The $G$-distributed random variable $(\xi,\eta)$
describes volatility and mean uncertainty of a probability model.
The corresponding convergence rate was established in \cite{Song2020} and
\cite{FPSS2019} using Stein's method, in \cite{Krylov2020} using stochastic
control method and in \cite{HL2020} using monotone approximation scheme method
under different model assumptions.
We refer the reader to \cite{C2016,Z2016,BK2022} and the references therein
for more research in this field.

A general $G$-L\'{e}vy process in the setting of sublinear expectation was
introduced by \cite{HP2009,HP2021} to further describe Poisson
jump uncertainty beyond volatility and mean uncertainty. They built a new type of
L\'{e}vy-Khintchine representation for $G$-L\'{e}vy processes by relating to a
class of fully nonlinear partial integro-differential equations (PIDEs).
Furthermore, the case of infinite activity jumps has been studied in
\cite{DKN2020,Kuhn2019,NN2017,NR2021}. An important class of nonlinear
L\'{e}vy processes with infinite activity jumps is the integrable $\alpha$-stable process
$(\zeta_{t})_{t\geq0}$ for $\alpha \in(1,2)$, which
corresponds to a fully nonlinear PIDE driven by a family of $\alpha$-stable
L\'{e}vy measures.
For the one-dimensional case,
the corresponding central limit theorem for the $\alpha$-stable process under
sublinear expectation was established in \cite{BM2016}, which shows that the
i.i.d. sequence $\{Z_{i}\}_{i=1}^{\infty}$ of \emph{real-valued} random
variables on a sublinear expectation space $(\Omega,\mathcal{H},\mathbb{\hat
{E})}$ converges in law to a nonlinear $\alpha$-stable distributed random
variable $\zeta$ under the integrability condition and an additional
consistency condition of $Z_{i}$, i.e.
\[
\lim_{n\rightarrow \infty}\mathbb{\hat{E}}\bigg[\phi \bigg(\frac{1}%
{\sqrt[\alpha]{n}}\sum_{i=1}^{n}Z_{i}\bigg)\bigg]=\mathbb{\tilde{E}}%
[\phi(\zeta)],\ \alpha\in(1,2),
\]
for any test function $\phi$. The corresponding convergence rate was
established in \cite{HJL2021} via a monotone approximation scheme.

Recently, together with Hu and Peng, the authors established a
universal robust limit theorem in \cite{HJLP2022} by further
developing the weak convergence approach introduced in
\cite{P2010_CLT}. This results covers Peng's robust central limit
theorem \cite{P20082,P2019} and Bayraktar-Munk's robust stable limit
theorem \cite{BM2016} as special cases. Under certain moment
conditions of $(X_{i},Y_{i},Z_{i})$ and a consistency condition of
$Z_{i}$, they proved that, for $\alpha \in(1,2)$, the normalized
sums of i.i.d. random variables $ \left \{  \left(
\frac{1}{\sqrt{n}}\sum_{i=1}^{n}X_{i},\frac{1}{n}\sum
_{i=1}^{n}Y_{i},\frac{1}{\sqrt[\alpha]{n}}\sum_{i=1}^{n}Z_{i}\right)
\right \}  _{n=1}^{\infty}%
$
converge in distribution to $\tilde{L}_{1}$, where $\tilde{L}_{t}=(\tilde{\xi
}_{t},\tilde{\eta}_{t},\tilde{\zeta}_{t})$, $t\in \lbrack0,1]$, is a
multidimensional nonlinear L\'{e}vy process with an uncertainty set $\Theta$
as a set of L\'{e}vy triplets. They further proved that the nonlinear L\'{e}vy
process can be characterized via a fully nonlinear and possibly degenerate
PIDE
\[
\left \{
\begin{array}
[c]{l}%
\displaystyle \partial_{t}u(t,x,y,z)-\sup \limits_{(F_{\mu},q,Q)\in \Theta
}\left \{  \int_{\mathbb{R}^{d}}\delta_{\lambda}u(t,x,y,z)F_{\mu}%
(d\lambda)\right. \\
\displaystyle \text{\  \  \  \  \  \  \  \  \  \  \  \  \  \  \  \ }\left.  +\langle
D_{y}u(t,x,y,z),q\rangle+\frac{1}{2}tr[D_{x}^{2}u(t,x,y,z)Q]\right \}  =0,\\
\displaystyle u(0,x,y,z)=\phi(x,y,z),\  \  \forall(t,x,y,z)\in \lbrack
0,1]\times \mathbb{R}^{3d},
\end{array}
\right.
\]
with $\delta_{\lambda}u(t,x,y,z):=u(t,x,y,z+\lambda)-u(t,x,y,z)-\langle
D_{z}u(t,x,y,z),\lambda \rangle$.

However, all the aforementioned works assume $\alpha \in(1,2)$,
leaving the more difficult regime $\alpha \in(0,1]$ open. There are
two main challenges. First, this case is outside the scope of the
integrability conditions of L\'{e}vy triplets in \cite{NN2017}, and
the PIDE characterization of the nonlinear $\alpha$-stable
distribution is unknown to date. This leads to the existing method
in \cite{BM2016} invalid, because \cite{BM2016} heavily relies on
the regularity estimates of fully nonlinear PIDEs. Second, due to
the lack of integrability for the normalized sums of i.i.d. random
variables, the sufficient condition in \cite{HJLP2022} for the
existence of their weak convergence limit fails.


We overcome the aforementioned difficulty by introducing a sublinear
expectation space with random variables not necessarily having any moments.
Hence, the random variables do not need to belong to the space but their
composition of bounded Lipschitz continuous functions do. This is sufficient
to define their distributions and study the corresponding robust central limit
theorems. We introduce a new $\delta$-moment condition on the normalized sums
of i.i.d. random variables (see the assumption (A.1) in Section 3) to
guarantee the tightness of the sequence so it has a weak convergence limit. With an additional consistency
condition (see the assumption (A.2) in Section 3), we show that the limiting
process is a nonlinear symmetric $\alpha$-stable process $(\tilde{\zeta}%
_{t})_{t\in \lbrack0,1]}$ such that
\[
\lim_{n\rightarrow \infty}\mathbb{\hat{E}}\bigg[\phi \bigg(\frac{1}%
{\sqrt[\alpha]{n}}\sum_{i=1}^{n}Z_{i}\bigg)\bigg]=\mathbb{\tilde{E}}%
[\phi(\tilde{\zeta}_{1})],\ \alpha\in(0,1],
\]
for any test function $\phi$. One remarkable feature of this process is that
$\tilde{\zeta}_{t}$ is non-integrable, i.e., $\mathbb{\tilde{E}}[|\tilde
{\zeta}_{1}|]=\infty$, so one cannot simply generalize the
L\'{e}vy-Khintichine representation formula in \cite{HJLP2022} to $\alpha
\in(0,1]$. We overcome this difficulty by using a truncation technique to
estimate the $\alpha$-stable L\'{e}vy measure, which gives a new type
L\'{e}vy-Khintichine representation for $(\tilde{\zeta}_{t})_{t\in \lbrack
0,1]}$ with $\alpha \in(0,1]$.

Together with the case $\alpha \in(1,2)$ studied in \cite{HJLP2022}, we are
able to give a complete characterization of the nonlinear and non-Gaussian
$\alpha$-stable process, for $\alpha \in(0,2)$, via a fully nonlinear PIDE
\[
\left \{
\begin{array}
[c]{l}%
\displaystyle \partial_{t}u(t,x)-\sup \limits_{F_{\mu}\in \mathcal{L}}\left \{
\int_{\mathbb{R}^{d}}\delta_{\lambda}^{\alpha}u(t,x)F_{\mu}(d\lambda)\right \}
=0,\\
\displaystyle u(0,x)=\phi(x),\text{ \  \ }\forall(t,x)\in \lbrack0,1]\times
\mathbb{R}^{d},
\end{array}
\right.
\]
where $\mathcal{L}$ is an $\alpha$-stable jump uncertainty set and
\[
\delta_{\lambda}^{\alpha}u(t,x):=\left \{
\begin{array}
[c]{ll}%
u(t,x+\lambda)-u(t,x)-\langle D_{x}u(t,x),\lambda \rangle, & \alpha \in(1,2),\\
u(t,x+\lambda)-u(t,x)-\langle D_{x}u(t,x),\lambda\mathbbm{1}_{\{|\lambda
|\leq1\}}\rangle, & \alpha=1,\\
u(t,x+\lambda)-u(t,x), & \alpha \in(0,1).
\end{array}
\right.
\]
Thanks to the connection with the above fully nonlinear PIDE, we also obtain
the existence of its viscosity solution as a byproduct, i.e. the solution can
be obtained via the weak convergence limit of the normalized sums of i.i.d.
random variables. When restricting the jump uncertainty
set to a singleton, our work also complements the classical $\alpha$-stable
central limit theorems in the linear setting. See, for example,
\cite{CX2019,DN2002,H19811,IL1971,JP1998,KK2001} and the references therein.

The article is organized as follows. In Section 2, we review some necessary
results about sublinear expectation for the $\alpha$-stable central limit
theorem without integrability condition. Section 3 details our main result,
discusses its connection with the classical linear case, and gives an example
highlighting the application of our main result. The proof of the main theorem
is given in Section 4 with some technical results provided in the Appendix.

\section{Preliminaries on nonlinear $\alpha$-stable process for $\alpha
\in(0,1]$}

The sublinear expectation framework has been introduced in
\cite{P2007,P20081,P2010_CLT,P2010}. However, the framework requires that
random variables defined on a sublinear expectation space are integrable,
which is too restrictive to study robust limit theorems for $\alpha$-stable
distribution with $\alpha \in(0,1]$.

To resolve the issue, we introduce a triplet $(\Omega,\mathcal{H}%
,\mathbb{\hat{E})}$, where $\Omega$ is a given set and $\mathcal{H}$ is a
linear space of real valued functions on $\Omega$ such that $\varphi
(X_{1},\ldots,X_{n})\in \mathcal{H}$ if $X_{1},\ldots,X_{n}\in \mathcal{H}$ for
each $\varphi \in{C_{b,Lip}}(\mathbb{R}^{n})$, the space of bounded and
Lipschitz continuous functions on $\mathbb{R}^{n}$. Then, $\mathbb{\hat{E}}$:
$\mathcal{H}\rightarrow \mathbb{R}$ is called a sublinear expectation if, for
all $X,Y\in \mathcal{H}$, it satisfies the following properties:

\begin{description}
\item[(i)] (Monotonicity) $\mathbb{\hat{E}}[X] \geq \mathbb{\hat{E}}[Y]$, if
$X\geq Y$;

\item[(ii)] (Constant preservation) $\mathbb{\hat{E}}[c] =c$, for
$c\in \mathbb{R}$;

\item[(iii)] (Sub-additivity) $\mathbb{\hat{E}}[X+Y] \leq \mathbb{\hat{E}} [X]
+\mathbb{\hat{E}}[Y] ;$

\item[(iv)] (Positive homogeneity) $\mathbb{\hat{E}}[\lambda X] =\lambda
\mathbb{\hat{E}}[X]$, for $\lambda>0$.
\end{description}

By a{n $n$-dimensional }random variable $X$ defined on $(\Omega,\mathcal{H}%
,\mathbb{\hat{E})}$, we mean $\varphi(X)\in \mathcal{H}$ for all
$\varphi \in C_{b,Lip}(\mathbb{R}^{n})$. However, the random
variable $X$ itself is not required to be in $\mathcal{H}$, i.e.
$\mathbb{\hat{E}}[X]$ may not exist. Our new framework relaxes the
integrability condition imposed on $X$ but it is sufficient to
define distribution of $X$.


\begin{definition}
Let {$X$ be a given $n$-dimensional random variable defined on a
sublinear expectation space $(\Omega,\mathcal{H},\mathbb{\hat{E})}$.
Define a functional on $C_{b,Lip}(\mathbb{R}^{n})$ by
\[
\mathbb{F}_{X}[\varphi]:=\mathbb{\hat{E}}[\varphi(X)],\text{ for }\varphi \in
C_{b,Lip}(\mathbb{R}^{n}).
\]
Then, $(\mathbb{R}^{n},C_{b,Lip}(\mathbb{R}^{n}),\mathbb{F}_{X})$ forms a
sublinear expectation space. $\mathbb{F}_{X}$ is called the distribution of
$X$ under $\mathbb{\hat{E}}$.}
\end{definition}

\begin{definition}
\label{converge in distribution}A sequence of $n$-dimensional random variables
$\{X_{i}\}_{i=1}^{\infty}$ defined on a sublinear expectation space
$(\Omega,\mathcal{H},\mathbb{\hat{E}})$ is said to converge in distribution
(or converge in law) under $\mathbb{\hat{E}}$ if for each $\varphi \in
C_{b,Lip}(\mathbb{R}^{n})$, the sequence $\{ \mathbb{F}_{X_{i}}[\varphi
]\}_{i=1}^{\infty}$ is a Cauchy sequence.

\end{definition}

Two random variables $X$ and $Y$, which may be defined on different sublinear
expectation spaces, are called \emph{identically distributed} if
$\mathbb{\hat{E}}_{1}[\varphi(X)] =\mathbb{\hat{E}}_{2}[\varphi(Y)]$, for all
$\varphi \in C_{b,Lip}(\mathbb{R}^{n})$. It is denoted by $X\overset{d}{=}Y$.
On the other hand, $Y$ is called \emph{independent} from $X$ if for every
$\varphi \in C_{b,Lip}(\mathbb{R}^{m}\times \mathbb{R}^{n})$ we have
\[
\mathbb{\hat{E}}\left[  \varphi(X,Y)\right]  =\mathbb{\hat{E}}\left[
\mathbb{\hat{E}}\left[  \varphi(x,Y)\right]  _{x=X}\right]  ,
\]
which is denoted by $Y\perp \! \! \! \perp X$. Notably, $Y$ is independent from
$X$ does not necessarily imply that $X$ is independent from $Y$. Then, $Y$ is
said to be an independent copy of $X$ if $Y\overset{d}{=}X$ and $Y\perp \! \!
\! \perp X$.


One can readily obtain the following two results concerning the convergence in
distribution of random variables.

\begin{proposition}
\label{independent copy converge in law}{Let $\{X_{i}\}_{i=1}^{\infty}$ be a
sequence of $n$-dimensional random variables defined on a sublinear
expectation spaces $(\Omega,\mathcal{H},\mathbb{\hat{E}})$. If $\{X_{i}%
\}_{i=1}^{\infty}$ converges in distribution to $X$ under $\mathbb{\hat{E}}$,
i.e.,
\[
\lim_{i\rightarrow \infty}\mathbb{\hat{E}}[\varphi(X_{i})]=\mathbb{\tilde{E}%
}[\varphi(X)],\text{ for }\varphi \in C_{b,Lip}(\mathbb{R}^{n}),
\]
which is denoted by $X_{i}\overset{\mathcal{D}}{\rightarrow}X$. Then,
\[
\bar{X}_{i}\overset{\mathcal{D}}{\rightarrow}\bar{X}\text{\ and \ }X_{i}%
+\bar{X}_{i}\overset{\mathcal{D}}{\rightarrow}X+\bar{X},
\]
where $\bar{X}_{i}$ and $\bar{X}$ are independent copy of $X_{i}$ and $X$,
respectively.}
\end{proposition}

\begin{remark}
\label{ramark i.i.d.}The above results can be generalized to multiple
summations. For each $i$ and $m\in \mathbb{N}$, let $\{X_{i}^{n}\}_{n=1}^{m}$
be an independent copy sequence of $X_{i}$ in the sense that $X_{i}%
^{1}\overset{d}{=}X_{i}$, $X_{i}^{n+1}\overset{d}{=}X_{i}^{n}$ and
$X_{i}^{n+1}\perp \! \! \! \perp(X_{i}^{1},X_{i}^{2},\ldots,X_{i}^{n})$ for
$n=1,\ldots,m-1$, and let $\{X^{n}\}_{n=1}^{m}$ be an independent copy
sequence of $X$ in the sense that $X^{1}\overset{d}{=}X$, $X^{n+1}\overset
{d}{=}X^{n}$ and $X^{n+1}\perp \! \! \! \perp(X^{1},X^{2},\ldots,X^{n})$ for
$n=1,\ldots,m-1$. If $X_{i}\overset{\mathcal{D}}{\rightarrow}X$, then
\[
\sum_{n=1}^{m}X_{i}^{n}\overset{\mathcal{D}}{\rightarrow}\sum_{n=1}^{m}X^{n}.
\]

\end{remark}

We are interested in non-integrable $\alpha$-stable distributed random
variables, and they serve as the attractors of normalized sums of i.i.d.
random variables.

\begin{definition}
{Let $\alpha \in(0,1]$. An $n$-dimensional random variable $X$ is said to be
(strictly) $\alpha$-stable under a sublinear expectation space $(\Omega
,\mathcal{H},\mathbb{\hat{E})}$ if}
\[
aX+bY\overset{d}{=}(a^{\alpha}+b^{\alpha})^{1/\alpha}X,\text{\ for }a,b\geq0,
\]
where $Y$ is an independent copy of $X$.
\end{definition}

\begin{remark}
A special attention is deserved for the case $\alpha=1$. If $|X|\in
{\mathcal{H}}$, then $\mathbb{\hat{E}}[|X|]$ exists. In this situation, $X$ is
the maximum distribution as introduced in \cite{P20082,P2010}. However, if
$|X|\notin \mathcal{H}$ but $\varphi(X)\in \mathcal{H}$ as considered in this
paper, then $\mathbb{\hat{E}}[|X|]$ does not exist but $\mathbb{\hat{E}%
}[\varphi(X)]$ exists. In this case $X$ is fundamentally different from the
maximum distribution.

\end{remark}

We close this section by recalling nonlinear L\'{e}vy processes under
sublinear expectations introduced in \cite{HP2021} and \cite{NN2017} and, in
particular, nonlinear symmetric $\alpha$-stable processes.

\begin{definition}
An $n$-dimensional c\`{a}dl\`{a}g process $(X_{t})_{t\geq0}$ defined on a
sublinear expectation space $(\Omega,\mathcal{H},\mathbb{\hat{E})}$ is called
a nonlinear L\'evy process if the following properties hold.

\begin{description}
\item[(i)] $X_{0}=0$;

\item[(ii)] $(X_{t})_{t\geq0}$ has stationary increments, that is,
$X_{t}-X_{s}$ and $X_{t-s}$ are identically distributed for all $0\leq s\leq
t;$

\item[(iii)] $(X_{t})_{t\geq0}$ has independent increments, that is,
$X_{t}-X_{s}$ is independent from $(X_{t_{1}},\ldots,X_{t_{n}})$ for each
$n\in \mathbb{N}$ and $0\leq t_{1}\leq \cdots \leq t_{n}\leq s\leq t$.
\end{description}

Furthermore, $X$ is called a nonlinear symmetric $\alpha$-stable process if
${X}_{t}\overset{d}{=}t^{1/\alpha}X_{1}$ and $X_{t}\overset{d}{=}-X_{t}$.
\end{definition}

\section{Main results}


Let $\alpha \in(0,1]$, $(\underline{\Lambda},\overline{\Lambda})$ for some
$\underline{\Lambda},\overline{\Lambda}>0$, and $F_{\mu}$ be the $\alpha
$-stable L\'{e}vy measure on $(\mathbb{R}^{d},\mathcal{B}(\mathbb{R}^{d}))$,
\begin{equation}
F_{\mu}(B)=\int_{S}\mu(dz)\int_{0}^{\infty}\mathbbm{1}_{B}(rz)\frac
{dr}{r^{1+\alpha}},\text{ \ for }B\in \mathcal{B}(\mathbb{R}^{d}), \label{F_mu}%
\end{equation}
where $\mu$ is a symmetric finite measure (cf. Sato \cite{Sato1999}) on the
unit sphere $S=\{z\in \mathbb{R}^{d}:|z|=1\}$. Introduce a \emph{jump
uncertainty set}:
\begin{equation}
\mathcal{L}=\left \{  F_{\mu}\  \text{measure on }\mathbb{R}^{d}:\mu
(S)\in(\underline{\Lambda},\overline{\Lambda})\right \}  . \label{L_0}%
\end{equation}
Let $\{Z_{i}\}_{i=1}^{\infty}$ be an i.i.d. non-integrable sequence of
$\mathbb{R}^{d}$-valued random variables defined on a sublinear expectation
space $(\Omega,\mathcal{H},\mathbb{\hat{E}})$ in the sense that $\mathbb{\hat
{E}}[|Z_{i}|]=\infty$, $Z_{i+1}$ $\overset{d}{=}Z_{i}$ and $Z_{i+1}$ is
independent from $(Z_{1},\ldots,Z_{i})$ for each $i\in \mathbb{N}$.

We impose the following assumptions throughout the paper.

\begin{description}
\item[(A1)] ($\delta$-Moment condition)
For $S_{n}:=\sum \limits_{i=1}^{n}Z_{i}$, $M_{\delta}:=\sup \limits_{n}%
\mathbb{\hat{E}}[|n^{-\frac{1}{\alpha}}S_{n}|^{\delta}]<\infty$, for some
$0<\delta<\alpha$.

\item[(A2)] (Consistency condition) For each $\varphi \in C_{b}^{3}%
(\mathbb{R}^{d})$, the space of functions on $\mathbb{R}^{d}$ with
uniformly bounded derivatives up to the order $3$, satisfies
\[
\frac{1}{s}\bigg \vert \mathbb{\hat{E}}\big[\varphi(x+s^{\frac{1}{\alpha}%
}Z_{1})-\varphi(x)\big]-s\sup \limits_{F_{\mu}\in \mathcal{L}}\int
_{\mathbb{R}^{d}}\delta_{\lambda}^{\alpha}\varphi(x)F_{\mu}(d\lambda
)\bigg \vert \leq l(s)\rightarrow0
\]
uniformly on $x\in \mathbb{R}^{d}$ as $s\rightarrow0$, where $l$ is a function
on $[0,1]$ and%
\[
\delta_{\lambda}^{\alpha}\varphi(x):=\left \{
\begin{array}
[c]{ll}%
\varphi(x+\lambda)-\varphi(x)-\langle D\varphi(x),\lambda
\mathbbm{1}_{\{|\lambda|\leq1\}}\rangle, & \alpha=1,\\
\varphi(x+\lambda)-\varphi(x), & \alpha \in(0,1).
\end{array}
\right.
\]


\end{description}

\subsection{Discussion on the $\delta$-moment and consistency conditions
(A1)-(A2)\label{Remark}}

The $\delta$-moment condition (A1) is inspired by the existing moment
conditions imposed in the literature. Indeed, for the robust central limit
theorem, the second moment condition $\mathbb{\hat{E}}[|Z_{i}|^{2}]<\infty$ is
equivalent to $M_{2}<\infty$ with $\alpha=2$, and
for the robust $\alpha$-stable central limit theorem in
\cite{HJLP2022} with $\alpha \in(1,2)$, the moment condition therein
is $M_{1}<\infty$. For $\alpha \in(0,1]$, since in this case the
$\alpha$-stable distributions are non-integrable, the
$\delta$-moment condition $M_{\delta}<\infty$ for some
$0<\delta<\alpha$ is natural. While the assumption (A1) guarantees
the weak convergence of the i.i.d. sequence
$\{Z_{i}\}_{i=1}^{\infty}$, the assumption (A2) is essentially a
consistency condition for the distribution of $Z_{i}$. The
consistency condition has been widely used in the numerical analysis
of the monotone approximation schemes for nonlinear PDEs
\cite{BJ2002,BJ2005,BJ2007}.

In the following, we will show that when our attention is confined
to the classical linear case, the $\delta$-moment condition and
consistency condition turn out to be mild and are more general than
the usual assumptions imposed in the literature. For simplicity, we
consider the one-dimensional case. Let $\zeta \ $be a classical
symmetric $\alpha$-stable random variable with L\'{e}vy triplet
$(F_{\mu},0,0)$. In this case, the symmetric finite measure $\mu$ in
$F_{\mu}$ is supported on $\{1,-1\}$ with $k:=\mu \{1\}=\mu \{-1\}$.
Theorem 2.6.7 from Ibragimov and Linnik \cite{IL1971} indicates that
\[
\frac{1}{\sqrt[\alpha]{n}}\sum_{i=1}^{n}Z_{i}\overset{\mathcal{D}}%
{\rightarrow}\zeta \text{, \ as}\ n\rightarrow \infty
\]
if $Z_{1}$ has the cumulative distribution function
\begin{equation}
F_{Z_{1}}(x)=\left \{
\begin{array}
[c]{ll}%
\displaystyle \left[  k/\alpha+\beta_{1}(x)\right]  \frac{1}{|x|^{\alpha}}, &
x<0,\\
\displaystyle1-\left[  k/\alpha+\beta_{2}(x)\right]  \frac{1}{x^{\alpha}}, &
x>0,
\end{array}
\right.  \label{cdf}%
\end{equation}
where $\beta_{1}:$ $(-\infty,0]$ $\rightarrow \mathbb{R}$ and $\beta
_{2}:[0,\infty)\rightarrow \mathbb{R}$ are functions satisfying
\[
\lim_{x\rightarrow-\infty}\beta_{1}(x)=\lim_{x\rightarrow \infty}\beta
_{2}(x)=0.
\]
{Under the condition (\ref{cdf}), it follows that } {{ }}%
\[
E[|Z_{1}|]\geq \int_{1}^{\infty}P(|Z_{1}|>z)dz=\frac{2k}{\alpha}\int
_{1}^{\infty}z^{-\alpha}dz+\int_{1}^{\infty}\frac{\beta_{1}(-z)+\beta_{2}%
(z)}{z^{\alpha}}dz=\infty.
\]
\  \  \

To verify the $\delta$-moment condition (A1), for given $n>0$ and
$0<\delta<\alpha$, we define an approximation scheme
$u_{n}:[0,1]\times \mathbb{R}\rightarrow \mathbb{R}$ recursively by
\[%
\begin{array}
[c]{l}%
u_{n}(t,z)=|z|^{\delta},\text{ \ if\ }t\in \lbrack0,1/n),\\
u_{n}(t,z)=E[u_{n}(t-1/n,z+n^{-1/\alpha}Z_{1})]\text{, \ if\ }t\in
\lbrack1/n,1].
\end{array}
\]
 We further assume that
\begin{equation}
\beta_{i},i=1,2,\text{are continuously differentiable functions
defined on
}(-\infty,0]\text{ and\ }[0,\infty)\text{, respectively.} \label{beta assum1}%
\end{equation}
Then, $u_{n}(1,0)=n^{-\frac{\delta}{\alpha}}E[|S_{n}|^{\delta}]$ and
$u_{n}(0,0)=0$. Using the regularity estimates in Appendix
\ref{Appendix 2} with $\mathcal{L}$ being a singleton, it can be
readily verified that (A1) holds.

Next, we will show that (A2) also holds. For this, we further assume that there exist
some constants $C>0$ and $\gamma \geq0$ such that
\begin{equation}
|\beta_{i}(x)|\leq \frac{C}{|x|^{\gamma}}\text{, \ }x\neq0.
\label{beta assum2.1}%
\end{equation}
The above condition also appears in \cite{CX2019} under the
classical linear case. In the case of $\alpha=1$, we need to further
impose that
\begin{equation}
\beta_{1}(x)=\beta_{2}(-x)\text{, \ for }x<0. \label{beta assum2.2}%
\end{equation}

\begin{proposition}
Under the classical linear expectation, suppose that the conditions
(\ref{cdf})-(\ref{beta assum2.1}) hold. Then, (A2) holds for $\alpha \in(0,1)$.
If we further assume that (\ref{beta assum2.2}) holds, then (A2) holds for
$\alpha=1$.
\end{proposition}

\begin{proof}
\emph{(i) }$\alpha \in(0,1)$. For $\varphi \in C_{b}^{3}(\mathbb{R})$ and
$s\in \lbrack0,1]$, from (\ref{cdf}), by changing variables, we can derive
that
\begin{align*}
&  \frac{1}{s}\bigg \vert E\big[\varphi(x+s^{\frac{1}{\alpha}}Z_{1}%
)-\varphi(x)\big]-s\int_{\mathbb{R}}\delta_{\lambda}^{\alpha}\varphi(x)F_{\mu
}(d\lambda)\bigg \vert \\
&  =\bigg \vert \int_{\mathbb{-\infty}}^{0}\delta_{\lambda}^{\alpha}%
\varphi(x)\big[\alpha \beta_{1}(s^{-\frac{1}{\alpha}}\lambda)-\beta_{1}%
^{\prime}(s^{-\frac{1}{\alpha}}\lambda)s^{-\frac{1}{\alpha}}\lambda
\big]|\lambda|^{-\alpha-1}d\lambda \  \\
&  \text{ \  \ }+\int_{0}^{\infty}\delta_{\lambda}^{\alpha}\varphi
(x)\big[\alpha \beta_{2}(s^{-\frac{1}{\alpha}}\lambda)-\beta_{2}^{\prime
}(s^{-\frac{1}{\alpha}}\lambda)s^{-\frac{1}{\alpha}}\lambda \big]\lambda
^{-\alpha-1}d\lambda \bigg \vert
\end{align*}
for all $x\in \mathbb{R}^{d}$. We only consider the integral above along the
positive half-line, and similarly for the integral along the negative
half-line. Set
\begin{align*}
&  \bigg \vert \int_{0}^{\infty}\delta_{\lambda}^{\alpha}\varphi
(x)[\alpha \beta_{2}(s^{-\frac{1}{\alpha}}\lambda)-\beta_{2}^{\prime}%
(s^{-\frac{1}{\alpha}}\lambda)s^{-\frac{1}{\alpha}}\lambda]\lambda^{-\alpha
-1}d\lambda \bigg \vert \\
&  =\bigg \vert \int_{0}^{1}\delta_{\lambda}^{\alpha}\varphi(x)\big[\alpha
\beta_{2}(s^{-\frac{1}{\alpha}}\lambda)-\beta_{2}^{\prime}(s^{-\frac{1}%
{\alpha}}\lambda)s^{-\frac{1}{\alpha}}\lambda \big]\lambda^{-\alpha-1}%
d\lambda \bigg \vert \\
&  \text{ \  \ }+\bigg \vert \int_{1}^{\infty}\delta_{\lambda}^{\alpha}%
\varphi(x)\big[\alpha \beta_{2}(s^{-\frac{1}{\alpha}}\lambda)-\beta_{2}%
^{\prime}(s^{-\frac{1}{\alpha}}\lambda)s^{-\frac{1}{\alpha}}\lambda
\big]\lambda^{-\alpha-1}d\lambda \bigg \vert:=\mathcal{I}+\mathcal{II}\text{.}%
\end{align*}
For the part $\mathcal{I}$, using integration by parts and the dominated
convergence theorem, one gets
\begin{align*}
\mathcal{I}  &  =\bigg \vert-\delta_{1}^{\alpha}\varphi(x)\beta_{2}%
(s^{-\frac{1}{\alpha}})+\int_{0}^{1}D\varphi(x+\lambda)\beta_{2}(s^{-\frac
{1}{\alpha}}\lambda)\lambda^{-\alpha}d\lambda \bigg \vert \\
&  \leq2\left \Vert \varphi \right \Vert _{\infty}|\beta_{2}(s^{-\frac{1}{\alpha
}})|+\left \Vert D\varphi \right \Vert _{\infty}\int_{0}^{1}|\beta_{2}%
(s^{-\frac{1}{\alpha}}\lambda)|\lambda^{-\alpha}d\lambda \rightarrow0\text{,
\ as }s\rightarrow0\text{.}%
\end{align*}
For the part $\mathcal{II}$, when $\gamma \in(1-\alpha,\infty)$, it follows
from integration by parts and (\ref{beta assum2.1}) that%
\begin{align*}
\mathcal{II}  &  =\bigg \vert \delta_{1}^{\alpha}\varphi(x)\beta_{2}%
(s^{-\frac{1}{\alpha}})+\int_{1}^{\infty}D\varphi(x+\lambda)\beta
_{2}(s^{-\frac{1}{\alpha}}\lambda)\lambda^{-\alpha}d\lambda \bigg \vert \\
&  \leq2\left \Vert \varphi \right \Vert _{\infty}|\beta_{2}(s^{-\frac{1}{\alpha
}})|+\left \Vert D\varphi \right \Vert _{\infty}\int_{1}^{\infty}|\beta
_{2}(s^{-\frac{1}{\alpha}}\lambda)|\lambda^{-\alpha}d\lambda \rightarrow
0\text{, \ as }s\rightarrow0\text{.}%
\end{align*}
When $\gamma \in \lbrack0,1-\alpha]$, we choose some $N_{0}>1$ such that
$|\beta_{2}(x)|\leq C$, for $|x|\geq N_{0}$. Then, it follows that
\begin{align*}
&  \bigg \vert \int_{N_{0}}^{\infty}\delta_{\lambda}^{\alpha}\varphi
(x)\big[\alpha \beta_{2}(s^{-\frac{1}{\alpha}}\lambda)-\beta_{2}^{\prime
}(s^{-\frac{1}{\alpha}}\lambda)s^{-\frac{1}{\alpha}}\lambda \big]\lambda
^{-\alpha-1}d\lambda \bigg \vert \\
&  \leq2\left \Vert \varphi \right \Vert _{\infty}\int_{N_{0}}^{\infty
}\big \vert \alpha \beta_{2}(s^{-\frac{1}{\alpha}}\lambda)-\beta_{2}^{\prime
}(s^{-\frac{1}{\alpha}}\lambda)s^{-\frac{1}{\alpha}}\lambda \big \vert \lambda
^{-\alpha-1}d\lambda \\
&  =2\left \Vert \varphi \right \Vert _{\infty}|\beta_{2}(s^{-\frac{1}{\alpha}%
}N_{0})|N_{0}^{-\alpha}\leq2C\left \Vert \varphi \right \Vert _{\infty}%
N_{0}^{-\alpha},
\end{align*}
and
\begin{align*}
&  \bigg \vert \int_{1}^{N_{0}}\delta_{\lambda}^{\alpha}\varphi(x)\big[\alpha
\beta_{2}(s^{-\frac{1}{\alpha}}\lambda)-\beta_{2}^{\prime}(s^{-\frac{1}%
{\alpha}}\lambda)s^{-\frac{1}{\alpha}}\lambda \big]\lambda^{-\alpha-1}%
d\lambda \bigg \vert \\
&  =\bigg \vert \delta_{1}^{\alpha}\varphi(x)\beta_{2}(s^{-\frac{1}{\alpha}%
})-\delta_{N_{0}}^{\alpha}\varphi(x)\beta_{2}(s^{-\frac{1}{\alpha}}N_{0}%
)N_{0}^{-\alpha}+\int_{1}^{N_{0}}D\varphi(x+\lambda)\beta_{2}(s^{-\frac
{1}{\alpha}}\lambda)\lambda^{-\alpha}d\lambda \bigg \vert \\
&  \leq2\left \Vert \varphi \right \Vert _{\infty}|\beta_{2}(s^{-\frac{1}{\alpha
}})|+2C\left \Vert \varphi \right \Vert _{\infty}N_{0}^{-\alpha}+\frac
{1}{1-\alpha}\left \Vert D\varphi \right \Vert _{\infty}\sup_{|\lambda|\geq
s^{-1/\alpha}}|\beta_{2}(\lambda)|N_{0}^{1-\alpha}.
\end{align*}
By letting $N_{0}=(\sup \limits_{|\lambda|\geq s^{-1/\alpha}}|\beta_{2}%
(\lambda)|)^{-1}$, we obtain that%
\[
\mathcal{II}\leq2\left \Vert \varphi \right \Vert _{\infty}|\beta_{2}%
(s^{-\frac{1}{\alpha}})|+\left[  4C\left \Vert \varphi \right \Vert _{\infty
}+\frac{1}{1-\alpha}\left \Vert D\varphi \right \Vert _{\infty}\right]
\sup \limits_{|\lambda|\geq s^{-1/\alpha}}|\beta_{2}(\lambda)|^{\alpha
}\rightarrow0\text{, \ as }s\rightarrow0\text{.}%
\]
\emph{(ii) }$\alpha=1$. We further impose that $\beta_{1}(x)=\beta_{2}(-x)$
for $x<0$. Then, for any $a>0$, we have $E[Z_{1}\mathbbm{1}_{\{|Z_{1}|\leq
a\}}]=0$, which implies that for $\varphi \in C_{b}^{3}(\mathbb{R})$ and
$s\in \lbrack0,1]$,
\begin{align*}
&  \frac{1}{s}\bigg \vert E\big[\varphi(x+sZ_{1})-\varphi(x)\big]-s\int
_{\mathbb{R}}\delta_{\lambda}^{1}\varphi(x)F_{\mu}(d\lambda)\bigg \vert \\
&  =\frac{1}{s}\bigg \vert E\big[\varphi(x+sZ_{1})-\varphi(x)-D\varphi
(x)sZ_{1}\mathbbm{1}_{\{|sZ_{1}|\leq1\}}\big]-s\int_{\mathbb{R}}%
\delta_{\lambda}^{1}\varphi(x)F_{\mu}(d\lambda)\bigg \vert \\
&  =\bigg \vert \int_{\mathbb{-\infty}}^{0}\delta_{\lambda}^{1}\varphi
(x)\big[\alpha \beta_{2}(-s^{-1}\lambda)+\beta_{2}^{\prime}(-s^{-1}%
\lambda)s^{-1}\lambda \big]|\lambda|^{-\alpha-1}d\lambda \\
&  \text{ \  \ }+\int_{0}^{\infty}\delta_{\lambda}^{1}\varphi(x)\big[\alpha
\beta_{2}(s^{-1}\lambda)-\beta_{2}^{\prime}(s^{-1}\lambda)s^{-1}%
\lambda \big]\lambda^{-\alpha-1}d\lambda \bigg \vert,
\end{align*}
for all $x\in \mathbb{R}^{d}$. Similar to the above process, consider the
integral above along the positive half-line, and denote
\begin{align*}
&  \bigg \vert \int_{0}^{\infty}\delta_{\lambda}^{1}\varphi(x)\big[\beta
_{2}(s^{-1}\lambda)-\beta_{2}^{\prime}(s^{-1}\lambda)s^{-1}\lambda
\big]\lambda^{-2}d\lambda \bigg \vert \\
&  =\bigg \vert \int_{0}^{1}\delta_{\lambda}^{1}\varphi(x)\big[\beta
_{2}(s^{-1}\lambda)-\beta_{2}^{\prime}(s^{-1}\lambda)s^{-1}\lambda
\big]\lambda^{-2}d\lambda \bigg \vert \\
&  \text{ \  \ }+\bigg \vert \int_{1}^{\infty}\delta_{\lambda}^{1}%
\varphi(x)\big[\beta_{2}(s^{-1}\lambda)-\beta_{2}^{\prime}(s^{-1}%
\lambda)s^{-1}\lambda \big]\lambda^{-2}d\lambda \bigg \vert:=\mathcal{I}%
+\mathcal{II}\text{.}%
\end{align*}
Then, we can deduce that
\begin{align*}
\mathcal{I}  &  =\bigg \vert-\delta_{1}^{1}\varphi(x)\beta_{2}(s^{-1}%
)+\int_{0}^{1}(D\varphi(x+\lambda)-D\varphi(x))\beta_{2}(s^{-1}\lambda
)\lambda^{-1}d\lambda \bigg \vert \\
&  \leq(2\left \Vert \varphi \right \Vert _{\infty}+\left \Vert D\varphi
\right \Vert _{\infty})|\beta_{2}(s^{-1})|+\left \Vert D^{2}\varphi \right \Vert
_{\infty}\int_{0}^{1}|\beta_{2}(s^{-1}\lambda)|d\lambda \rightarrow0\text{,
\ as }s\rightarrow0\text{,}%
\end{align*}
and
\[
\mathcal{II}\leq2\left \Vert \varphi \right \Vert _{\infty}|\beta_{2}%
(s^{-1})|+\left \Vert D\varphi \right \Vert _{\infty}\int_{1}^{\infty}|\beta
_{2}(s^{-1}\lambda)|\lambda^{-1}d\lambda \rightarrow0\text{, \ as }%
s\rightarrow0\text{.}%
\]
To sum up, the assumption (A2) holds.
\end{proof}

\subsection{Robust $\alpha$-stable central limit theorem}

We are ready to state the main result of this paper, which is dubbed as
\emph{a robust }$\alpha$\emph{-stable central\ limit theorem} under sublinear expectation.

\begin{theorem}
\label{main theorem} Suppose that the assumptions (A1)-(A2) hold.\ Then, there
exists a nonlinear L\'{e}vy process $(\tilde{\zeta}_{t})_{t\in \lbrack0,1]}$,
associated with the jump uncertainty set $\mathcal{L}$ such that for any
$\phi \in C_{b,Lip}(\mathbb{R}^{d})$,
\[
\lim_{n\rightarrow \infty}\mathbb{\hat{E}}\left[  \phi \left(  \frac{S_{n}%
}{\sqrt[\alpha]{n}}\right)  \right]  =\mathbb{\tilde{E}}[\phi(\tilde{\zeta
}_{1})]=u^{\phi}(1,0),
\]
where $u^{\phi}$ is the unique viscosity solution of the following fully
nonlinear PIDE
\begin{equation}
\left \{
\begin{array}
[c]{l}%
\displaystyle \partial_{t}u(t,x)-\sup \limits_{F_{\mu}\in \mathcal{L}}\left \{
\int_{\mathbb{R}^{d}}\delta_{\lambda}^{\alpha}u(t,x)F_{\mu}(d\lambda)\right \}
=0,\\
\displaystyle u(0,x)=\phi(x),\text{\  \  \ }\forall(t,x)\in \lbrack
0,1]\times \mathbb{R}^{d},
\end{array}
\right.  \label{PIDE}%
\end{equation}
where
\[
\delta_{\lambda}^{\alpha}u(t,x):=\left \{
\begin{array}
[c]{ll}%
u(t,x+\lambda)-u(t,x)-\langle D_{x}u(t,x),\lambda \mathbbm{1}_{\{|\lambda
|\leq1\}}\rangle, & \alpha=1,\\
u(t,x+\lambda)-u(t,x), & \alpha \in(0,1).
\end{array}
\right.
\]
Furthermore, the limiting process $(\tilde{\zeta}_{t})_{t\in \lbrack0,1]}$ is a
nonlinear symmetric $\alpha$-stable process, i.e. $\tilde{\zeta}_{t}%
\overset{d}{=}t^{1/\alpha}\tilde{\zeta}_{1}$ and $\tilde{\zeta}_{t}\overset
{d}{=}-\tilde{\zeta}_{t}$, for any $0\leq t\leq1$.
\end{theorem}

\subsection{{An example}\label{Example}}

We discuss a concrete example to illustrates how to construct the i.i.d.
sequence $\{Z_{i}\}_{i=1}^{\infty}$ and their corresponding sublinear
expectation space, as well as the rationality of the assumptions (A1)-(A2).

For simplicity, we consider the case $d=1$. Given $\underline{\Lambda
},\overline{\Lambda}>0$. Let $F_{\mu}$ be the L\'{e}vy measure with the
symmetric measure $\mu$ concentrated on the points $S_{0}=\{1,-1\}$ and
\[
\mathcal{L}=\left \{  F_{\mu}\  \text{measure on }\mathbb{R}^{2}:\mu(\chi
)\in(\underline{\Lambda},\overline{\Lambda}),\  \text{for }\chi \in
S_{0}\right \}  .
\]
Denote $K=(\underline{\Lambda},\overline{\Lambda})$ and $k=\mu \{1\}=\mu
\{-1\}$. Let $0<\delta<\alpha$ and $f(n)$ be a non-negative function on
$\mathbb{N}$ tending to 0 as $n\rightarrow \infty$. For each $k\in K$, let
$W_{k}$ be a classical random variable such that

\begin{description}
\item[(i)] $W_{k}$ has a cumulative distribution function
\[
F_{W_{k}}(x)=\left \{
\begin{array}
[c]{ll}%
\displaystyle \left[  k/\alpha+\beta_{k}(-x)\right]  \frac{1}{|x|^{\alpha}}, &
x<0,\\
\displaystyle1-\left[  k/\alpha+\beta_{k}(x)\right]  \frac{1}{x^{\alpha}}, &
x>0,
\end{array}
\right.
\]
where $\beta_{k}:[0,\infty)\rightarrow \mathbb{R}$ are continuously
differentiable functions such that there exist $\gamma \geq0$ and $C>0$
satisfying
\[
|\beta_{k}(x)|\leq \frac{C}{|x|^{\gamma}}\text{, \ }x\neq0.
\]

\item[(ii)] {the following quantities are less than $f(n)$\ for all $n$:}
\[%
\begin{array}
[c]{lll}%
\displaystyle|\beta_{k}(n^{1/\alpha})|,\text{ \ } & \displaystyle \int
_{1}^{\infty}\frac{|\beta_{k}(n^{1/\alpha}x)|}{x^{1+\alpha-\delta}}dx,\text{
\ } & \displaystyle \int_{0}^{1}\frac{|\beta_{k}(n^{1/\alpha}x)|}{x^{\alpha}%
}dx.
\end{array}
\]

\end{description}

{Let $\Omega=\mathbb{R}$ and $\mathcal{H}_{0}$ be the space of bounded and
continuous functions on $\mathbb{R}$. For each $X=f(x)\in \mathcal{H}_{0}$,
define the sublinear expectation}
\[
\mathbb{\hat{E}}[X]=\sup_{k\in K}\int_{\mathbb{R}}f(x)dF_{W_{k}}(x).
\]
Denote by $\mathcal{H}$\ the completion of $\mathcal{H}_{0}$ under the
norm$\  \left \Vert X\right \Vert :=\mathbb{\hat{E}}[|X|]$. $\mathbb{\hat
{E}[\cdot]}$ can be extended continuously to $\mathcal{H}$, on which it is a
sublinear expectation and is still denoted as $\big(\Omega,\mathcal{H}
,\mathbb{\hat{E}}\big)$, see \cite{DHP2011} and \cite{P2010} for a detailed discussion.

Consider an $\mathbb{R}$-valued random variable $Z(z)=z$, $z\in \mathbb{R}$.
Clearly, $Z^{\delta,m}:=|Z|^{\delta}\wedge m\in \mathcal{H}_{0}$\ for\ any
$m>0$. We claim that $\left \{  Z^{\delta,m}\right \}  _{m=1}^{\infty}$ is a
Cauchy sequence, which yields that $\mathbb{\hat{E}}[|Z|^{\delta}%
]:=\lim \limits_{m\rightarrow \infty}\mathbb{\hat{E}}[Z^{\delta,m}]$ and
$|Z|^{\delta}\in \mathcal{H}_{0}$. Indeed, for any $Z^{\delta,m}$,
$Z^{\delta,n}\in \mathcal{H}$, we have%
\begin{align*}
\mathbb{\hat{E}}[|Z^{\delta,m}-Z^{\delta,n}|] &  \leq \mathbb{\hat{E}}\left[
\left \vert |Z|^{\delta}\wedge m-|Z|^{\delta}\right \vert \right]
+\mathbb{\hat{E}}\left[  \left \vert |Z|^{\delta}\wedge n-|Z|^{\delta
}\right \vert \right]  \\
&  \leq \sup_{k\in K}\int_{\mathbb{R}}|z|^{\delta}\mathbbm{1}_{\{|z|^{\delta
}>m\}}dF_{W_{k}}(z)+\sup_{k\in K}\int_{\mathbb{R}}|z|^{\delta}%
\mathbbm{1}_{\{|z|^{\delta}>n\}}dF_{W_{k}}(z)\\
&  =2m^{\frac{\delta-\alpha}{\delta}}\sup_{k\in K}\left \{  \frac{k}%
{\alpha-\delta}+\delta \int_{1}^{\infty}\frac{\beta_{k}(m^{1/\delta}%
z)}{z^{1+\alpha-\delta}}dz+\beta_{k}(m^{1/\delta})\right \}  \\
&  \text{ \  \ }+2n^{\frac{\delta-\alpha}{\delta}}\sup_{k\in K}\left \{
\frac{k}{\alpha-\delta}+\delta \int_{1}^{\infty}\frac{\beta_{k}(n^{1/\delta}%
z)}{z^{1+\alpha-\delta}}dz+\beta_{k}(n^{1/\delta})\right \}
\end{align*}
tending to 0, as $m,n\rightarrow \infty$. A similar calculation shows that $|Z|\notin \mathcal{H}$.

Next, we will verify the assumption (A1). Let $\{Z_{i}\}_{i=1}^{\infty}$ be a
sequence of i.i.d. $\mathbb{R}$-valued random variables in the sense that
$Z_{1}\overset{d}{=}Z$, $Z_{i+1}\overset{d}{=}Z_{i}$, and $Z_{i+1}%
\perp \! \! \! \perp(Z_{1},Z_{2},\ldots,Z_{i})$ for each $i\in \mathbb{N}$. For
given $n>0$, define an approximation scheme $u_{n}:[0,1]\times \mathbb{R}%
\rightarrow \mathbb{R}$ recursively by
\begin{equation}%
\begin{array}
[c]{l}%
u_{n}(t,z)=|z|^{\delta},\text{ \ if\ }t\in \lbrack0,1/n),\\
u_{n}(t,z)=\mathbb{\hat{E}}[u_{n}(t-1/n,z+n^{-1/\alpha}Z_{1})]\text{,
\ if\ }t\in \lbrack1/n,1].
\end{array}
\label{2.2}%
\end{equation}
Then, $u_{n}(1,0)=n^{-\frac{\delta}{\alpha}}\mathbb{\hat{E}}[|S_{n}|^{\delta
}]$ and $u_{n}(0,0)=0$. For any $t,s\in \lbrack0,1]$ and $z\in \mathbb{R}$, we
deduce the following estimate, whose proof is postponed to the Appendix
\ref{Appendix 2} due to its technicality
\[
\left \vert u_{n}(t,z)-u_{n}(s,z)\right \vert \leq I_{n}(|t-s|^{\delta
/2}+n^{-\delta/2}),
\]
where $I_{n}<\infty$ as $n\rightarrow \infty$. It follows that
\[
u_{n}(1,0)=u_{n}(1,0)-u_{n}(0,0)<\infty \text{, \ as }n\rightarrow \infty,
\]
which implies (A1) holds.

For the assumption (A2), following along similar arguments as in Proposition
\ref{Remark}, we can derive that for $\varphi \in C_{b}^{3}(\mathbb{R})$,
\[
\frac{1}{s}\bigg \vert \mathbb{\hat{E}}\big[\varphi(z+s^{\frac{1}{\alpha}%
}Z_{1})-\varphi(z)\big]-s\sup_{F_{\mu}\in \mathcal{L}}\int_{\mathbb{R}^{2}%
}\delta_{\lambda}^{\alpha}\varphi(z)F_{\mu}(d\lambda)\bigg \vert \rightarrow
0,
\]
uniformly on $z\in \mathbb{R}$ as $s\rightarrow0$.

\begin{remark}
\label{Remark Pareto}Assume that $\mathcal{L}=\{F_{\mu}\}$ is a singleton and
$\{Z_{i}\}_{i=1}^{\infty}$ is a sequence of i.i.d. random variables with a
Pareto distribution in \cite{DN2002,JP1998,KK2001}, i.e., the probability
density is
\[
p(x)=\frac{\alpha}{2|x|^{\alpha+1}}\mathbbm{1}_{[1,\infty)}(|x|).
\]
It is easy to check that the above distribution satisfies the conditions
(i)-(ii). Theorem \ref{main theorem} implies that $\frac{1}{\sqrt[\alpha]{n}%
}S_{n}$\ converges in distribution to $\tilde{\zeta}_{1}$, as $n\rightarrow
\infty$, where $\tilde{\zeta}_{1}$ is a classical symmetric $\alpha$-stable
distribution with L\'{e}vy triplet $(F_{\mu},0,0)$.
\end{remark}

\section{Proof of Theorem \ref{main theorem}}

\label{Proof of Main Theorem}

\subsection{Outline of the proof}

For the reader's convenience, we outline the main steps below with the
detailed proof provided in the rest of this section.

\textbf{(i)} By using the notions of tightness and weak compactness, we first
construct a nonlinear L\'{e}vy process $(\tilde{\zeta}_{t})_{t\in[0,1]}$ on
some sublinear expectation space $(\tilde{\Omega},Lip(\tilde{\Omega
}),\mathbb{\tilde{E}})$ in Section
\ref{Section_The construction of Levy process}, which is generated by the weak
convergence limit of the sequences $\big \{(t/n)^{1/\alpha}S_{n}%
\big \}_{n=1}^{\infty}$ for $t\in \lbrack0,1]$.

\textbf{(ii)} To link the nonlinear L\'{e}vy process $(\tilde{\zeta}%
_{t})_{t\in \lbrack0,1]}$ with the fully nonlinear PIDE (\ref{PIDE}), a key
step is to give the characterization of
\[
\lim \limits_{s\rightarrow0}\mathbb{\tilde{E}}[\varphi(\tilde{\zeta}%
_{s})]s^{-1}\text{ for }\varphi \in C_{b}^{3}(\mathbb{R}^{d})\text{ with
}\varphi(0)=0.
\]
It follows from a new estimate for the $\alpha$-stable L\'{e}vy measure and
the squeeze theorem in Section \ref{Section_Representation of Levy process}.

\textbf{(iii)} Once the representation of the nonlinear L\'{e}vy process is
established, with the help of nonlinear stochastic analysis techniques and
viscosity solution methods, Theorem \ref{main theorem} is a consequence of the
dynamic programming principle in Section \ref{Section_connection PIDE} by
defining $u(t,x)=\mathbb{\tilde{E}}[\phi(x+\tilde{\zeta}_{t})]$, for
$(t,x)\in \lbrack0,1]\times \mathbb{R}^{d}$.

\textbf{(iv)} Finally, we show the scaling and symmetric properties of
$(\tilde{\zeta}_{t})_{t\in[0,1]}$ in Section
\ref{Section_Scaling and symmetric properties}, so it is indeed a nonlinear
symmetric $\alpha$-stable process.

\subsection{The construction of the nonlinear $\alpha$-stable process}

\label{Section_The construction of Levy process}

Let $\tilde{\Omega}=D_{0}^{d}[0,1]$ be the space of all $\mathbb{R}^{d}%
$-valued paths $(\omega_{t})_{t\in \lbrack0,1]}$ with $\omega_{0}=0$, equipped
with the Skorohod topology, where $D_{0}^{d}[0,1]$ is the space of
$\mathbb{R}^{d}$-valued c\`{a}dl\`{a}g paths. Consider the canonical process
$\tilde{\zeta}_{t}(\omega)=\omega_{t}$, $t\in \lbrack0,1]$, for $\omega
\in \tilde{\Omega}$. Set
\[
Lip(\tilde{\Omega})=\left \{  \varphi(\tilde{\zeta}_{t_{1}},\ldots,\tilde
{\zeta}_{t_{n}}-\tilde{\zeta}_{t_{n-1}}):\forall0\leq t_{1}<t_{2}<\cdots
<t_{n}\leq1,\varphi \in C_{b,Lip}(\mathbb{R}^{d\times n})\right \}  .
\]

We have the following theorem, which can be regarded as a Donsker theorem for
the nonlinear\ L\'{e}vy process $(\tilde{\zeta}_{t})_{t\in \lbrack0,1]}$.

\begin{theorem}
\label{The construction of Levy process}Assume that (A1) holds. Then, there
exists a sublinear expectation $\mathbb{\tilde{E}}$ on $(\tilde{\Omega
},Lip(\tilde{\Omega}))$ such that the sequence $\{n^{-1/\alpha}S_{n}%
\}_{n=1}^{\infty}$ converges in distribution to $\tilde{\zeta}_{1}$, where
$(\tilde{\zeta}_{t})_{t\in \lbrack0,1]}$ is a nonlinear L\'{e}vy process on
$(\tilde{\Omega},Lip(\tilde{\Omega}), \mathbb{\tilde{E}})$.
\end{theorem}

The proof of the above theorem depends on the following lemma.

\begin{lemma}
\label{tight}Assume that (A1) holds. For $\phi \in C_{b,Lip}(\mathbb{R}^{d})$,
let
\[
\mathbb{\hat{F}}[\phi]:=\sup_{n}\mathbb{\hat{E}}\left[  \phi \left(
\frac{S_{n}}{\sqrt[\alpha]{n}}\right)  \right]  .
\]
Then, the sublinear expectation $\mathbb{\hat{F}}$ on $(\mathbb{R}%
^{d},C_{b,Lip}(\mathbb{R}^{d}))$ is tight in the sense of Definition
\ref{def_2.7} in the Appendix.
\end{lemma}

\begin{proof}
It is clear that $\mathbb{\hat{F}}$ is a sublinear expectation on
$(\mathbb{R}^{d},C_{b,Lip}(\mathbb{R}^{d}))$. Now we show that $\mathbb{\hat
{F}}$ is tight. For any $N>0$, define
\[
\varphi_{N}(x)=\left \{
\begin{array}
[c]{ll}%
1, & |x|>N,\\
|x|-N+1, & N-1\leq|x|\leq N,\\
0, & |x|<N-1.
\end{array}
\right.
\]
One can easily check that $\varphi_{N}\in C_{b,Lip}(\mathbb{R})$ and
\[
\mathbbm{1}_{\{|x|>N\}}\leq \varphi_{N}(x)\leq \mathbbm{1}_{\{|x|>N-1\}}%
\leq \frac{|x|^{\delta}}{(N-1)^{\delta}},
\]
for $0<\delta<\alpha$. Thus, for each $\varepsilon>0$, we can choose
$N_{0}>\sqrt[\delta]{M_{\delta}/\varepsilon}+1$ such that
\[
\mathbb{\hat{F}}[\varphi_{N_{0}}(z)]=\sup_{n}\mathbb{\hat{E}}\left[
\varphi_{N_{0}}\left(  \frac{S_{n}}{\sqrt[\alpha]{n}}\right)  \right]
\leq \frac{M_{\delta}}{(N_{0}-1)^{\delta}}<\varepsilon,
\]
where we have used the assumption (A1). This proves the desired result.
\end{proof}

\begin{proof}
[Proof of Theorem \ref{The construction of Levy process}]We denote $\bar
{S}_{n}=n^{-1/\alpha}S_{n}$. Seeing that $\mathbb{\hat{F}}$ is tight and
\[
\mathbb{\hat{E}}[\phi(\bar{S}_{n})]-\mathbb{\hat{E}}[\phi^{\prime}(\bar{S}%
_{n})]\leq \mathbb{\hat{F}}[\phi-\phi^{\prime}],\text{ for }\phi,\phi^{\prime
}\in C_{b,Lip}(\mathbb{R}^{d}),
\]
by Corollary \ref{corollary tight theorem}, there exists a subsequence $\{
\bar{S}_{n_{i}}\}_{i=1}^{\infty}\subset \{ \bar{S}_{n}\}_{n=1}^{\infty}$ which
converges in distribution to some $\zeta_{1}$ in $(\Omega,\mathcal{H}%
,\mathbb{\hat{E}}_{1})$. For the above convergent subsequence $\{ \bar
{S}_{n_{i}}\}_{i=1}^{\infty}$, it is clear that for an arbitrarily increasing
integers of $\{ \tilde{n}_{i}\}_{i=1}^{\infty}$ such that $|\tilde{n}%
_{i}-n_{i}|\leq1$, both $\{ \bar{S}_{n_{i}}\}_{i=1}^{\infty}$ and $\{ \bar
{S}_{\tilde{n}_{i}}\}_{i=1}^{\infty}$ converges in distribution to the same
limit. Thus, without loss of generality, we assume that $n_{i}$,
$i=1,2,\ldots$, are all even numbers and decompose into two parts:
\[
\bar{S}_{n_{i}}=\frac{1}{\sqrt[\alpha]{2}}(n_{i}/2)^{-\frac{1}{\alpha}%
}S_{n_{i}/2}+\frac{1}{\sqrt[\alpha]{2}}(n_{i}/2)^{-\frac{1}{\alpha}}(S_{n_{i}%
}-S_{n_{i}/2}):=\  \bar{S}_{n_{i}/2}^{1/2}+(\bar{S}_{n_{i}}-\bar{S}_{n_{i}%
/2}^{1/2}),
\]
where $\bar{S}_{n}^{t}:=(t/n)^{1/\alpha}S_{n}$ for $t\in \lbrack0,1)$. For the
first part, applying the same argument again, we prove that there exists a
subsequence $\big \{ \bar{S}_{n_{i}^{1}/2}^{1/2}\big \}_{i=1}^{\infty}$
$\subset \big \{ \bar{S}_{n_{i}/2}^{1/2}\big \}_{i=1}^{\infty}\ $such that
$\big \{ \bar{S}_{n_{i}^{1}/2}^{1/2}\big \}_{i=1}^{\infty}$ converging in
distribution to $\zeta_{1/2}$. Since $\bar{S}_{n_{i}^{1}}-\bar{S}_{n_{i}%
^{1}/2}^{1/2}$ is an independent copy of $\bar{S}_{n_{i}^{1}/2}^{1/2}$, by
Proposition \ref{independent copy converge in law}, we know that
\[
\bar{S}_{n_{i}^{1}}-\bar{S}_{n_{i}^{1}/2}^{1/2}\overset{\mathcal{D}%
}{\rightarrow}\bar{\zeta}_{1/2}\text{ \ and \ }\bar{S}_{n_{i}^{1}}%
\overset{\mathcal{D}}{\rightarrow}\zeta_{1/2}+\bar{\zeta}_{1/2},
\]
where $\bar{\zeta}_{1/2}$ is an independent copy of $\zeta_{1/2}$. In
addition, $\bar{S}_{n_{i}^{1}}\overset{\mathcal{D}}{\rightarrow}\zeta_{1}$.
Thus $\zeta_{1}\overset{d}{=}\zeta_{1/2}+\bar{\zeta}_{1/2}.$

Repeating the previous procedure for $\bar{S}_{n_{i}^{1}/2}^{1/2}$, we can
define random variable $\zeta_{1/4}$. Proceeding in this way, one can obtain
$\zeta_{1/2^{m}}$ in $(\Omega,\mathcal{H},\mathbb{\hat{E}}_{1})$,
$m\in \mathbb{N}$, such that for each $\zeta_{1/2^{m}}$ there exists a
convergent sequence $\{ \bar{S}_{n_{i}^{m}/2^{m}}^{1/2^{m}}\}_{i=1}^{\infty}$
converging in distribution to it. Finally, using the random variables
$\zeta_{1/2^{m}}(m\in \mathbb{N)}$, we may construct a sublinear expectation
$\mathbb{\tilde{E}}$ on $(\tilde{\Omega},Lip(\tilde{\Omega}))$ such that the
canonical process $(\tilde{\zeta}_{t})_{t\in \lbrack0,1]}$ is a nonlinear
L\'{e}vy process.

Indeed, we will construct a sublinear expectation $\mathbb{\tilde{E}%
}:Lip(\tilde{\Omega})\rightarrow \mathbb{R}$ such that the canonical process
$(\tilde{\zeta}_{t})_{t\in \lbrack0,1]}$ is a nonlinear L\'{e}vy process. For
each $m\geq0$, let $\tau_{m}=2^{-m}$,%
\[%
\begin{array}
[c]{l}%
\displaystyle \mathcal{H}^{m}=\big \{ \varphi \big(\tilde{\zeta}_{\tau_{m}%
},\tilde{\zeta}_{2\tau_{m}}-\tilde{\zeta}_{\tau_{m}},\ldots,\tilde{\zeta
}_{2^{m}\tau_{m}}-\tilde{\zeta}_{(2^{m}-1)\tau_{m}}\big):\forall \varphi \in
C_{b,Lip}\big(\mathbb{R}^{d\times2^{m}}\big)\big \} \text{, \ for }m\geq1,\\
\displaystyle \mathcal{H}^{0}=\big \{ \phi \big(\tilde{\zeta}_{1}%
\big):\forall \phi \in C_{b,Lip}\big(\mathbb{R}^{d}\big)\big \}.
\end{array}
\]
Let $\{ \zeta_{\tau_{m}}^{n}\}_{n=1}^{\infty}$ be a sequence of i.i.d.
$\mathbb{R}^{d}$-valued random variables on $(\Omega,\mathcal{H}%
,\mathbb{\hat{E}}_{1})$ in the sense that $\zeta_{\tau_{m}}^{1}\overset{d}%
{=}\zeta_{\tau_{m}}$, $\zeta_{\tau_{m}}^{n+1}\overset{d}{=}\zeta_{\tau_{m}%
}^{n}$ and $\zeta_{\tau_{m}}^{n+1}\perp \! \! \! \perp(\zeta_{\tau_{m}}%
^{1},\zeta_{\tau_{m}}^{2},\ldots,\zeta_{\tau_{m}}^{n})$ for each
$n\in \mathbb{N}$. For given $m\geq1$, $\phi \big(\tilde{\zeta}_{n\tau_{m}%
}-\tilde{\zeta}_{(n-1)\tau_{m}}\big)$ with $1\leq n\leq2^{m}$, and $\phi \in
C_{b,Lip}(\mathbb{R}^{d})$, define
\[
\mathbb{\tilde{E}}^{m}\big[\phi \big(\tilde{\zeta}_{n\tau_{m}}-\tilde{\zeta
}_{(n-1)\tau_{m}}\big)\big]=\mathbb{\hat{E}}_{1}\big[\phi(\zeta_{\tau_{m}}%
^{n})\big].
\]
For $\varphi \big(\tilde{\zeta}_{\tau_{m}},\ldots,\tilde{\zeta}_{2^{m}\tau_{m}%
}-\tilde{\zeta}_{(2^{m}-1)\tau_{m}}\big)\in \mathcal{H}^{m}$, for some
$\varphi \in C_{b,Lip}\big(\mathbb{R}^{d\times2^{m}}\big)$, define
\[
\mathbb{\tilde{E}}^{m}\big[\varphi \big(\tilde{\zeta}_{\tau_{m}},\ldots
,\tilde{\zeta}_{2^{m}\tau_{m}}-\tilde{\zeta}_{(2^{m}-1)\tau_{m}}%
\big)\big]=\varphi_{0},
\]
where $\varphi_{0}$ is defined iteratively through
\[%
\begin{array}
[c]{c}%
\displaystyle \varphi_{2^{m}-1}(x_{1},x_{2},\ldots,x_{2^{m}-1})=\mathbb{\tilde
{E}}^{m}\big[\varphi \big(x_{1},x_{2},\ldots,x_{2^{m}-1},\tilde{\zeta}%
_{2^{m}\tau_{m}}-\tilde{\zeta}_{(2^{m}-1)\tau_{m}}\big)\big]\\
\displaystyle \varphi_{2^{m}-2}(x_{1},x_{2},\ldots,x_{2^{m}-2})=\mathbb{\tilde
{E}}^{m}\big[\varphi_{2^{m}-1}\big(x_{1},x_{2},\ldots,x_{2^{m}-2},\tilde
{\zeta}_{(2^{m}-1)\tau_{m}}-\tilde{\zeta}_{(2^{m}-2)\tau_{m}}\big)\big]\\
\displaystyle \vdots \\
\displaystyle \varphi_{1}(x_{1})=\mathbb{\tilde{E}}^{m}\big[\varphi
_{2}\big(x_{1},\tilde{\zeta}_{2\tau_{m}}-\tilde{\zeta}_{\tau_{m}}\big)\big]\\
\displaystyle \varphi_{0}=\mathbb{\tilde{E}}^{m}\big[\varphi_{1}%
\big(\tilde{\zeta}_{\tau_{m}}\big)\big].
\end{array}
\]
For $\phi \in C_{b,Lip}(\mathbb{R}^{d})$, we also define
\[
\mathbb{\tilde{E}}^{0}\big[\phi \big(\tilde{\zeta}_{1}\big)\big]=\mathbb{\hat
{E}}_{1}\big[\phi(\zeta_{1})\big].
\]
From the above definition we know that $(\tilde{\Omega},\mathcal{H}%
^{m},\mathbb{\tilde{E}}^{m})$ is a sublinear expectation space under which
$\tilde{\zeta}_{t}-\tilde{\zeta}_{s}\overset{d}{=}$ $\tilde{\zeta}_{t-s}$ and
$\tilde{\zeta}_{t}-\tilde{\zeta}_{s}\perp \! \! \! \perp(\tilde{\zeta}_{t_{1}%
},\ldots,\tilde{\zeta}_{t_{i}})$ for each $t_{i},s,t\in \mathcal{D}%
_{m}[0,1]:=\{l2^{-m}:0\leq l\leq2^{m},l\in \mathbb{N}\}$ with $t_{i}\leq s\leq
t$. Also, it can be checked that $\mathbb{\tilde{E}}^{m}[\cdot]$ is
consistent, i.e., for each $m\geq0$, $\mathbb{\tilde{E}}^{m+1}[\cdot
]=\mathbb{\tilde{E}}^{m}[\cdot]$ on $\mathcal{H}^{m}$.

Next, we denote
\[
\mathcal{H}^{\infty}=\bigcup_{m\geq0}\mathcal{H}^{m}.
\]
Obviously, $\mathcal{H}^{\infty}\subset Lip(\tilde{\Omega})$. For any $\chi
\in \mathcal{H}^{\infty}$, there exists an $m_{0}\in \mathbb{N}$ such that
$\chi \in \mathcal{H}^{m_{0}}$, define
\[
\mathbb{\tilde{E}}[\chi]:=\mathbb{\tilde{E}}^{m_{0}}[\chi].
\]
Since $\mathbb{\tilde{E}}^{m}[\cdot]$\ is consistent, $\mathbb{\tilde{E}%
}:\mathcal{H}^{\infty}\rightarrow \mathbb{R}$ is a well-defined sublinear
expectation.

We shall extend the sublinear expectation $\mathbb{\tilde{E}}:\mathcal{H}%
^{\infty}\rightarrow \mathbb{R}$ to $\mathbb{\tilde{E}}:Lip(\tilde{\Omega
})\rightarrow \mathbb{R}$ and still use $\mathbb{\tilde{E}}$ for simplicity.
Denote $\mathcal{D}_{\infty}[0,1]:=\cup_{m\geq0}\mathcal{D}_{m}[0,1]$. For
each $\varphi(\tilde{\zeta}_{t_{1}},\ldots,\tilde{\zeta}_{t_{n}}-\tilde{\zeta
}_{t_{n-1}})\in Lip(\tilde{\Omega})$ with $\varphi \in C_{b,Lip}(\mathbb{R}%
^{d\times n})$, for each $t_{k}\in \lbrack0,1]$, $1\leq k\leq n$, we choose a
sequence $\{t_{k}^{i}\}_{i=1}^{\infty}\in \mathcal{D}_{\infty}[0,1]$ such that
$t_{k}^{i}<t_{k+1}^{i}$ and $t_{k}^{i}\downarrow t_{k}$ as $i\rightarrow
\infty$. Define
\[
\mathbb{\tilde{E}}\big[\varphi \big(\tilde{\zeta}_{t_{1}},\ldots,\tilde{\zeta
}_{t_{n}}-\tilde{\zeta}_{t_{n-1}}\big)\big]=\lim_{i\rightarrow \infty
}\mathbb{\tilde{E}}\big[\varphi \big(\tilde{\zeta}_{t_{1}^{i}},\ldots
,\tilde{\zeta}_{t_{n}^{i}}-\tilde{\zeta}_{t_{n-1}^{i}}\big)\big].
\]
It can be verified that the limit does not depend on the choice of
$\{t_{k}^{i}\}_{i=1}^{\infty}$. Indeed, for two descending sequences
$\{t_{k}^{i}\}_{i=1}^{\infty}$ and $\{t_{k}^{i^{\prime}}\}_{i^{\prime}%
=1}^{\infty}$ with the same limit $t_{k}$ as $i,i^{\prime}\rightarrow \infty$,
we assume that $t_{k}^{i}-t_{k}^{i^{\prime}}:=l_{k}\tau_{m_{k}}$ for some
$m_{k}\in \mathbb{N}$ and $1\leq l_{k}\leq2^{m_{k}}$. From the construction of
$\tilde{\zeta}_{\cdot}$ and Remark \ref{ramark i.i.d.}, there exists a
convergent sequence $\{ \bar{S}_{n_{j}^{\ast}}^{\tau_{m_{k}}}\}_{j=1}^{\infty
}$ such that for $\varphi \in C_{b,Lip}(\mathbb{R}^{d\times n})$
\[%
\begin{array}
[c]{l}%
\displaystyle \left \vert \mathbb{\tilde{E}}\big[\varphi \big(\tilde{\zeta
}_{t_{1}^{i}},\ldots,\tilde{\zeta}_{t_{n}^{i}}-\tilde{\zeta}_{t_{n-1}^{i}%
}\big)\big]-\mathbb{\tilde{E}}\big[\varphi \big(\tilde{\zeta}_{t_{1}%
^{i^{\prime}}},\ldots,\tilde{\zeta}_{t_{n}^{i^{\prime}}}-\tilde{\zeta
}_{t_{n-1}^{i^{\prime}}}\big)\big]\right \vert \\
\displaystyle \leq L_{\varphi}\mathbb{\tilde{E}}\big[\big(\sum_{k=1}%
^{n}|\tilde{\zeta}_{t_{k}^{i}}-\tilde{\zeta}_{t_{k-1}^{i}}-\tilde{\zeta
}_{t_{k}^{i^{\prime}}}+\tilde{\zeta}_{t_{k-1}^{i^{\prime}}}|\big)\wedge
N_{\varphi}\big]\\
\displaystyle \leq2L_{\varphi}\sum_{k=1}^{n}\mathbb{\tilde{E}}\big[|\tilde
{\zeta}_{t_{k}^{i}-t_{k}^{i^{\prime}}}|\wedge N_{\varphi}\big]\\
\displaystyle=2L_{\varphi}\sum \limits_{k=1}^{n}\mathbb{\hat{E}}_{1}%
\big[|\zeta_{\tau_{m_{k}}}^{1}+\cdots+\zeta_{\tau_{m_{k}}}^{l_{k}}|\wedge
N_{\varphi}\big]\\
\displaystyle=2L_{\varphi}\sum_{k=1}^{n}\lim \limits_{j\rightarrow \infty
}\mathbb{\hat{E}}\big[|\bar{S}_{l_{k}n_{j}^{\ast}}^{t_{k}^{i}-t_{k}%
^{i^{\prime}}}|\wedge N_{\varphi}\big]\\
\displaystyle=2L_{\varphi}\sum_{k=1}^{n}\lim \limits_{j\rightarrow \infty
}\mathbb{\hat{E}}\big[\big|(t_{k}^{i}-t_{k}^{i^{\prime}})^{1/\alpha}%
(l_{k}n_{j}^{\ast})^{-1/\alpha}S_{l_{k}n_{j}^{\ast}}\big|\wedge N_{\varphi
}\big]\\
\displaystyle \leq2L_{\varphi}\sum_{k=1}^{n}|t_{k}^{i}-t_{k}^{i^{\prime}%
}|^{\delta/\alpha}\lim \limits_{j\rightarrow \infty}\mathbb{\hat{E}}%
\big[|(l_{k}n_{j}^{\ast})^{-1/\alpha}S_{l_{k}n_{j}^{\ast}}|^{\delta
}\big]N_{\varphi}^{1-\delta}\\
\displaystyle \leq2L_{\varphi}M_{\delta}N_{\varphi}^{1-\delta}\sum_{k=1}%
^{n}|t_{k}^{i}-t_{k}^{i^{\prime}}|^{\delta/\alpha}\\
\displaystyle \rightarrow0\text{, \ as }i,i^{\prime}\rightarrow \infty,
\end{array}
\]
where $L_{\varphi}>0$ is the Lipschitz constant of $\varphi$, $K_{\varphi
}:=\left \Vert \varphi \right \Vert _{\infty}$ and $N_{\varphi}:=\frac
{2K_{\varphi}}{L_{\varphi}}$. Also, $\mathbb{\tilde{E}}:Lip(\tilde{\Omega
})\rightarrow \mathbb{R}$ is a well-defined sublinear expectation, that is, if
$\varphi(\tilde{\zeta}_{t_{1}},\ldots,$ $\tilde{\zeta}_{t_{n}}-\tilde{\zeta
}_{t_{n-1}})=\varphi^{\prime}(\tilde{\zeta}_{t_{1}},\ldots,\tilde{\zeta
}_{t_{n}}-\tilde{\zeta}_{t_{n-1}})$ with $\varphi,\varphi^{\prime}\in
C_{b,Lip}(\mathbb{R}^{d\times n})$, then
\[
\mathbb{\tilde{E}}\big[\varphi \big(\tilde{\zeta}_{t_{1}},\ldots,\tilde{\zeta
}_{t_{n}}-\tilde{\zeta}_{t_{n-1}}\big)\big]=\mathbb{\tilde{E}}\big[\varphi
^{\prime}\big(\tilde{\zeta}_{t_{1}},\ldots,\tilde{\zeta}_{t_{n}}-\tilde{\zeta
}_{t_{n-1}}\big)\big].
\]
Moreover, for each $t_{i},s,t\in \lbrack0,1]\ $with $t_{i}\leq s\leq t$,
$\tilde{\zeta}_{t}-\tilde{\zeta}_{s}\overset{d}{=}$ $\tilde{\zeta}_{t-s}$ and
$\tilde{\zeta}_{t}-\tilde{\zeta}_{s}\perp \! \! \! \perp(\tilde{\zeta}_{t_{1}%
},\ldots,\tilde{\zeta}_{t_{i}})$ under $\mathbb{\tilde{E}}$. Thus,
$(\tilde{\Omega},Lip(\tilde{\Omega}),\mathbb{\tilde{E}})$ is a sublinear
expectation space on which the canonical process $(\tilde{\zeta}_{t}%
)_{t\in \lbrack0,1]}$ is a nonlinear L\'{e}vy process.

Finally, note that {the distribution of $\tilde{\zeta}_{1}$ is uniquely
determined by $u(1,0)$, where $u$ is the unique viscosity solution of the
fully nonlinear PIDE (\ref{1.0}) (see Theorem \ref{unique viscosity theorem}%
).} We complete the proof by further claiming that
\[
\lim_{n\rightarrow \infty}\mathbb{\hat{E}}[\phi(\bar{S}_{n})]=\mathbb{\tilde
{E}}[\phi(\tilde{\zeta}_{1})]\text{, \ for\ }\phi \in C_{b,Lip}(\mathbb{R}%
^{d}).
\]
Suppose not. We can find a sequence $\{ \bar{S}_{\tilde{n}}\}_{\tilde{n}
=1}^{\infty}\subset \{ \bar{S}_{n}\}_{n=1}^{\infty}$ such that for any
subsequence of $\{ \bar{S}_{\tilde{n}}\}_{\tilde{n}=1}^{\infty}$ does not
converge in distribution to $\tilde{\zeta}_{1}$. However, using Lemma
\ref{tight} for $\{ \bar{S}_{\tilde{n}}\}_{\tilde{n}=1}^{\infty}$, we derive
from the above process that there exists a subsequence $\{ \bar{S}_{\tilde
{n}_{i}}\}_{i=1}^{\infty} \subset \{ \bar{S}_{\tilde{n}}\}_{\tilde{n}%
=1}^{\infty}$ which converges in distribution to $\tilde{\zeta}_{1}$, which
induces a contradiction. The proof is completed.
\end{proof}

\subsection{L\'{e}vy-Khintchine representation of nonlinear $\alpha$-stable
process}

\label{Section_Representation of Levy process}

For $\alpha \in(1,2)$, the L\'{e}vy-Khintchine representation for a nonlinear
$\alpha$-stable process has been established in \cite{HJLP2022}. However, the
L\'{e}vy-Khintchine representation for a nonlinear $\alpha$-stable process
with $\alpha \in(0,1]$ is still lacking. It turns out such a representation is
crucial for the corresponding robust limit theorem.

Recall that $F_{\mu}$ is the $\alpha$-stable L\'{e}vy measure given in
(\ref{F_mu}), and $\mathcal{L}$ is the set of $\alpha$-stable L\'{e}vy
measures on $\mathbb{R}^{d}$ satisfying (\ref{L_0}). Denote
\begin{equation}
\mathcal{K}_{\alpha}:=\left \{
\begin{array}
[c]{ll}%
\sup \limits_{F_{\mu}\in \mathcal{L}}\int_{\mathbb{R}^{d}}|\lambda|^{2}%
\wedge1F_{\mu}(d\lambda), & \alpha=1,\\
\sup \limits_{F_{\mu}\in \mathcal{L}}\int_{\mathbb{R}^{d}}|\lambda|\wedge
1F_{\mu}(d\lambda), & \alpha \in(0,1).
\end{array}
\right.  \label{F_mu condition 1}%
\end{equation}
It follows from the Sato-type result (see \cite[Remark 14.4]{Sato1999}) that
$\mathcal{K}_{\alpha}<\infty.$

\begin{lemma}
\label{Lipschitz}For each $\varphi \in C_{b}^{3}(\mathbb{R}^{d})$, we have for
$x,x^{\prime}\in \mathbb{R}^{d}$,
\[
\sup \limits_{F_{\mu}\in \mathcal{L}}\int_{\mathbb{R}^{d}}\left \vert
\delta_{\lambda}^{\alpha}\varphi(x^{\prime})-\delta_{\lambda}^{\alpha}
\varphi(x) \right \vert F_{\mu}(d\lambda)\leq C_{\alpha}|x^{\prime}-x|^{\delta
},
\]
where
\[
C_{\alpha}=\left \{
\begin{array}
[c]{ll}%
\mathcal{K}_{\alpha}\big(  4\left \Vert D\varphi \right \Vert _{\infty}^{\delta
}\left \Vert \varphi \right \Vert _{\infty}^{1-\delta}+2\left \Vert D^{3}
\varphi \right \Vert _{\infty}^{\delta}\left \Vert D^{2}\varphi \right \Vert
_{\infty}^{1-\delta}\big), & \alpha=1,\\
\mathcal{K}_{\alpha}\big(  4\left \Vert D\varphi \right \Vert _{\infty}^{\delta
}\left \Vert \varphi \right \Vert _{\infty}^{1-\delta}+2\left \Vert D^{2}
\varphi \right \Vert _{\infty}^{\delta}\left \Vert D\varphi \right \Vert _{\infty
}^{1-\delta}\big), & \alpha \in(0,1).
\end{array}
\right.
\]

\end{lemma}

\begin{proof}
When $\alpha \in(0,1)$, we have for $x\in \mathbb{R}^{d}$,
\[
\delta_{\lambda}^{\alpha}\varphi(x) =\varphi \left(  x+\lambda \right)
-\varphi(x) =\int_{0}^{1}\langle D\varphi(x+\theta \lambda),\lambda \rangle
d\theta.
\]
Then, it follows that for $x,x^{\prime}\in \mathbb{R}^{d}$,%
\begin{align*}
&  \int_{|\lambda|\leq1}\left \vert \delta_{\lambda}^{\alpha}\varphi(x^{\prime
})-\delta_{\lambda}^{\alpha}\varphi(x) \right \vert F_{\mu}(d\lambda)\\
&  =\int_{|\lambda|\leq1}\left \vert \int_{0}^{1}\langle D\varphi(x^{\prime
}+\theta \lambda)-D\varphi(x+\theta \lambda),\lambda \rangle d\theta \right \vert
F_{\mu}(d\lambda)\\
&  \leq \int_{|\lambda|\leq1}|\lambda|F_{\mu}(d\lambda)\left \Vert D^{2}
\varphi \right \Vert _{\infty}\left(  |x^{\prime}-x|\wedge \frac{2\left \Vert
D\varphi \right \Vert _{\infty}}{\left \Vert D^{2}\varphi \right \Vert _{\infty}
}\right) \\
&  \leq \int_{|\lambda|\leq1}|\lambda|F_{\mu}(d\lambda)\left \Vert D^{2}%
\varphi \right \Vert _{\infty}^{\delta}(2\left \Vert D\varphi \right \Vert
_{\infty})^{1-\delta}|x^{\prime}-x|^{\delta},
\end{align*}
and
\begin{align*}
&  \int_{|\lambda|>1}\left \vert \delta_{\lambda}^{\alpha}\varphi(x^{\prime
})-\delta_{\lambda}^{\alpha}\varphi(x) \right \vert F_{\mu}(d\lambda)\\
&  =\int_{|\lambda|>1}\left \vert \varphi(x^{\prime}+\lambda)-\varphi
(x+\lambda)-(\varphi(x^{\prime})-\varphi(x))\right \vert F_{\mu}(d\lambda)\\
&  \leq2\int_{|\lambda|>1}F_{\mu}(d\lambda)\left \Vert D\varphi \right \Vert
_{\infty}\left(  |x^{\prime}-x|\wedge \frac{2\left \Vert \varphi \right \Vert
_{\infty}}{\left \Vert D\varphi \right \Vert _{\infty}}\right) \\
&  \leq2\int_{|\lambda|>1}F_{\mu}(d\lambda)\left \Vert D\varphi \right \Vert
_{\infty}^{\delta}(2\left \Vert \varphi \right \Vert _{\infty})^{1-\delta
}|x^{\prime}-x|^{\delta}.
\end{align*}
When $\alpha=1$, we have for $x\in \mathbb{R}^{d}$,
\[
\delta_{\lambda}^{1}\varphi(x) =\int_{0}^{1}\int_{0}^{1}\langle D^{2}%
\varphi(x+\tau \theta \lambda)\lambda,\lambda \rangle \theta d\tau d\theta \text{,
\ for }|\lambda|\leq1,
\]
and $\delta_{\lambda}^{1}\varphi(x) =\varphi \left(  x+\lambda \right)
-\varphi(x) $, for $|\lambda|>1$. Similarly, we can deduce that
\begin{align*}
&  \int_{|\lambda|\leq1}\left \vert \delta_{\lambda}^{1}\varphi(x^{\prime
})-\delta_{\lambda}^{1}\varphi(x) \right \vert F_{\mu}(d\lambda)\\
&  =\int_{|\lambda|\leq1}\left \vert \int_{0}^{1}\int_{0}^{1}\langle
(D^{2}\varphi(x^{\prime}+\tau \theta \lambda)-D^{2}\varphi(x+\tau \theta
\lambda))\lambda,\lambda \rangle \theta d\tau d\theta \right \vert F_{\mu
}(d\lambda)\\
&  \leq \int_{|\lambda|\leq1}|\lambda|^{2}F_{\mu}(d\lambda)\left \Vert
D^{3}\varphi \right \Vert _{\infty}^{\delta}(2\left \Vert D^{2}\varphi \right \Vert
_{\infty})^{1-\delta}|x^{\prime}-x|^{\delta},
\end{align*}
and
\[
\int_{|\lambda|>1}\left \vert \delta_{\lambda}^{1}\varphi(x^{\prime}%
)-\delta_{\lambda}^{1}\varphi(x) \right \vert F_{\mu}(d\lambda)\leq
2\int_{|\lambda|>1}F_{\mu}(d\lambda)\left \Vert D\varphi \right \Vert _{\infty
}^{\delta}(2\left \Vert \varphi \right \Vert _{\infty})^{1-\delta}|x^{\prime
}-x|^{\delta}.
\]
The proof is completed.
\end{proof}

The following estimate is crucial to our main result.

\begin{theorem}
\label{recursive}Assume that (A1)-(A2) hold. Then, for $\varphi \in C_{b}%
^{3}(\mathbb{R}^{d})$ and $s\in \lbrack0,1]$,
\begin{equation}
\lim_{n\rightarrow \infty}\bigg \vert \mathbb{\hat{E}}\bigg[\varphi \left(
x+(s/n)^{\frac{1}{\alpha}}S_{n}\right)  \bigg]-\varphi(x)-s\sup \limits_{F_{\mu
}\in \mathcal{L}}\int_{\mathbb{R}^{d}}\delta_{\lambda}^{\alpha}\varphi
(x)F_{\mu}(d\lambda)\bigg \vert=o(s),\nonumber
\end{equation}
uniformly on $x\in \mathbb{R}^{d}$, where $o(s)/s\rightarrow0$ as
$s\rightarrow0$.
\end{theorem}

\begin{proof}
Because $Z_{n}$ is independent from $Z_{1},\dots,Z_{n-1}$, we have
\begin{align*}
&  \mathbb{\hat{E}}\bigg[\varphi \left(  x+(s/n)^{\frac{1}{\alpha}}%
S_{n}\right)  \bigg]-\varphi(x)-s\epsilon(x)\\
&  =\mathbb{\hat{E}}\bigg[\left.  \mathbb{\hat{E}}\bigg[\varphi \left(
x+(s/n)^{\frac{1}{\alpha}}(\omega_{n-1}+Z_{n})\right)  \bigg]\right \vert
_{\substack{z_{1}=Z_{1}\\ \cdots \\z_{n-1}=Z_{n-1}}}\bigg]-s\epsilon
(x)-\varphi(x),
\end{align*}
where
\[
\omega_{n-1}:=\sum \limits_{k=1}^{n-1}z_{k}\  \  \  \text{and}\  \  \  \epsilon
(x):=\sup \limits_{F_{\mu}\in \mathcal{L}}\int_{\mathbb{R}^{d}}\delta_{\lambda
}^{\alpha}\varphi(x)F_{\mu}(d\lambda).
\]
Thanks to the assumptions (A1)-(A2) and Lemma \ref{Lipschitz}, we deduce that
\begin{align*}
&  \mathbb{\hat{E}}\bigg[\varphi \left(  x+(s/n)^{\frac{1}{\alpha}}%
(\omega_{n-1}+Z_{n})\right)  \bigg]\\
&  =\mathbb{\hat{E}}\bigg[\varphi \left(  x+(s/n)^{\frac{1}{\alpha}}%
(\omega_{n-1}+Z_{n})\right)  -\varphi \left(  x+(s/n)^{\frac{1}{\alpha}}%
\omega_{n-1}\right) \\
&  \text{ \  \ }-\frac{s}{n}\sup \limits_{F_{\mu}\in \mathcal{L}}\int
_{\mathbb{R}^{d}}\delta_{\lambda}^{\alpha}\varphi \left(  x+(s/n)^{\frac
{1}{\alpha}}\omega_{n-1}\right)  F_{\mu}(d\lambda)\bigg]\\
&  \text{ \  \ }+\frac{s}{n}\sup \limits_{F_{\mu}\in \mathcal{L}}\int
_{\mathbb{R}^{d}}\delta_{\lambda}^{\alpha}\varphi \left(  x+(s/n)^{\frac
{1}{\alpha}}\omega_{n-1}\right)  F_{\mu}(d\lambda)\\
&  \text{ \  \ }-\frac{s}{n}\sup \limits_{F_{\mu}\in \mathcal{L}}\int
_{\mathbb{R}^{d}}\delta_{\lambda}^{\alpha}\varphi \left(  x\right)  F_{\mu
}(d\lambda)+\varphi \left(  x+(s/n)^{\frac{1}{\alpha}}\omega_{n-1}\right)
+\frac{s}{n}\epsilon(x)\\
&  \leq \frac{s}{n}l\left(  \frac{s}{n}\right)  + C\left(  \frac{s}{n}\right)
^{1+\frac{\delta}{\alpha}}\left \vert \omega_{n-1}\right \vert ^{\delta}+
\varphi \left(  x+(s/n)^{\frac{1}{\alpha}}\omega_{n-1}\right)  +\frac{s}%
{n}\epsilon(x),
\end{align*}
which implies the following one-step estimate
\begin{align*}
&  \mathbb{\hat{E}}\left[  \varphi \left(  x+(s/n)^{\frac{1}{\alpha}}%
S_{n}\right)  \right] \\
&  \leq \mathbb{\hat{E}}\left[  \varphi \left(  x+(s/n)^{\frac{1}{\alpha}%
}S_{n-1}\right)  \right]  +\frac{s}{n}\epsilon(x)+\frac{s}{n}l\left(  \frac
{s}{n}\right)  +CM_{\delta}s^{1+\frac{\delta}{\alpha}}\frac{1}{n}\left(
\frac{n-1}{n}\right)  ^{\frac{\delta}{\alpha}}.
\end{align*}
Repeating the above process recursively, we obtain that
\[
\mathbb{\hat{E}}\left[  \varphi \left(  x+(s/n)^{\frac{1}{\alpha}}S_{n}\right)
\right]  \leq \varphi \left(  x\right)  +s\epsilon(x)+sl\left(  \frac{s}%
{n}\right)  +CM_{\delta}s^{1+\frac{\delta}{\alpha}}\frac{1}{n}\sum
\limits_{k=1}^{n-1}\left(  \frac{k}{n}\right)  ^{\frac{\delta}{\alpha}}.
\]
Analogously, we have
\[
\mathbb{\hat{E}}\left[  \varphi \left(  x+(s/n)^{\frac{1}{\alpha}}S_{n}%
^{3}\right)  \right]  \geq \varphi \left(  x\right)  +s\epsilon(x)-sl\left(
\frac{s}{n}\right)  -CM_{\delta}s^{1+\frac{\delta}{\alpha}}\frac{1}{n}%
\sum \limits_{k=1}^{n-1}\left(  \frac{k}{n}\right)  ^{\frac{\delta}{\alpha}}.
\]
Thus,
\[
\lim_{n\rightarrow \infty}\left \vert \mathbb{\hat{E}}\left[  \varphi \left(
x+(s/n)^{\frac{1}{\alpha}}S_{n}\right)  \right]  -\varphi(x)-s\epsilon
(x)\right \vert \leq CM_{\delta}s^{1+\frac{\delta}{\alpha}}\frac{\alpha}%
{\delta+\alpha},
\]
where we have used the fact that
\[
\lim_{n\rightarrow \infty}\frac{1}{n}\sum \limits_{k=1}^{n-1}\left(  \frac{k}%
{n}\right)  ^{\frac{\delta}{\alpha}}=\frac{\alpha}{\delta+\alpha}.
\]
This implies the desired result.
\end{proof}

Denote
\[
\mathfrak{F}=\left \{  \varphi \in C_{b}^{3}(\mathbb{R}^{d}):\varphi
(0)=0\right \}
\]
In the following, we shall present the characterization of $\lim
\limits_{s\rightarrow0}\mathbb{\tilde{E}}[\varphi(\tilde{\zeta}_{s})]s^{-1}$
for $\varphi \in \mathfrak{F}$, which can be regarded as a new type of
L\'{e}vy-Khintchine representation for the nonlinear pure jump L\'{e}vy
process $(\tilde{\zeta}_{t})_{t\in \lbrack0,1]}$. It will play an important
role in establishing the related PIDE in Section \ref{Section_connection PIDE}.

\begin{theorem}
\label{represent theorem}Assume that (A1)-(A2) hold. Then, for each
$\varphi \in \mathfrak{F}$,
\[
\lim \limits_{s\rightarrow0}\mathbb{\tilde{E}}[\varphi(\tilde{\zeta}%
_{s})]s^{-1}=\sup_{F_{\mu}\in \mathcal{L}}\left \{  \int_{\mathbb{R}^{d}}%
\delta_{\lambda}^{\alpha}\varphi(0)F_{\mu}(d\lambda)\right \}  .
\]

\end{theorem}

\begin{proof}
For each $\varphi \in \mathfrak{F}$ and $s\in \lbrack0,1]$,\ Theorem
\ref{The construction of Levy process} shows that there exists a sequence
$\{s_{k}\}_{k=1}^{\infty}\subset \mathcal{D}_{\infty}[0,1]$ satisfying
$s_{k}\downarrow s$ as $k\rightarrow \infty$ and a convergent sequence
$\big \{(s_{k}/n_{i}^{\ast})^{\frac{1}{\alpha}}S_{n_{i}^{\ast}}\big \}_{i=1}%
^{\infty}$ for each $s_{k}$, such that
\[
\mathbb{\tilde{E}}[\varphi(\tilde{\zeta}_{s})]=\lim_{k\rightarrow \infty}%
\lim_{i\rightarrow \infty}\mathbb{\hat{E}}\big[\varphi((s_{k}/n_{i}^{\ast
})^{\frac{1}{\alpha}}S_{n_{i}^{\ast}})\big].
\]
Under the assumption (A2), it follows from Theorem \ref{recursive} that
\begin{align}
&  \bigg \vert \mathbb{\tilde{E}}[\varphi(\tilde{\zeta}_{s})]-\varphi
(0)-s\sup \limits_{F_{\mu}\in \mathcal{L}}\int_{\mathbb{R}^{d}}\delta_{\lambda
}^{\alpha}\varphi(0)F_{\mu}(d\lambda)\bigg \vert \label{Z condition}\\
&  \leq \lim_{k\rightarrow \infty}\bigg \vert \lim_{i\rightarrow \infty
}\mathbb{\hat{E}}\big[\varphi((s_{k}/n_{i}^{\ast})^{\frac{1}{\alpha}}%
S_{n_{i}^{\ast}})\big]-\varphi(0)-s_{k}\sup \limits_{F_{\mu}\in \mathcal{L}}%
\int_{\mathbb{R}^{d}}\delta_{\lambda}^{\alpha}\varphi(0)F_{\mu}(d\lambda
)\bigg \vert=o(s),\nonumber
\end{align}
which implies that
\begin{align*}
\lim \limits_{s\rightarrow0}\mathbb{\tilde{E}}[\varphi(\tilde{\zeta}%
_{s})]s^{-1}  &  =\lim_{s\rightarrow0}\bigg(\mathbb{\tilde{E}}[\varphi
(\tilde{\zeta}_{s})]-\varphi(0)-s\sup \limits_{F_{\mu}\in \mathcal{L}}%
\int_{\mathbb{R}^{d}}\delta_{\lambda}^{\alpha}\varphi(0)F_{\mu}(d\lambda
)\bigg)s^{-1}+\sup \limits_{F_{\mu}\in \mathcal{L}}\int_{\mathbb{R}^{d}}%
\delta_{\lambda}^{\alpha}\varphi(0)F_{\mu}(d\lambda)\\
&  \leq \sup \limits_{F_{\mu}\in \mathcal{L}}\int_{\mathbb{R}^{d}}\delta
_{\lambda}^{\alpha}\varphi(0)F_{\mu}(d\lambda),
\end{align*}
and similarly,
\[
\lim \limits_{s\rightarrow0}\mathbb{\tilde{E}}[\varphi(\tilde{\zeta}%
_{s})]s^{-1}\geq \sup \limits_{F_{\mu}\in \mathcal{L}}\int_{\mathbb{R}^{d}}%
\delta_{\lambda}^{\alpha}\varphi(0)F_{\mu}(d\lambda).
\]
Therefore, by the\ squeeze theorem, we complete the proof.
\end{proof}

\subsection{Connection to PIDE}

\label{Section_connection PIDE}

In this section, we relate the nonlinear pure jump L\'{e}vy process
$(\tilde{\zeta}_{t})_{t\in \lbrack0,1]}$ to the fully nonlinear PIDE
(\ref{PIDE}). Let $C_{b}^{2,3}([0,1]\times \mathbb{R}^{d})$ denote the set of
functions on $[0,1]\times \mathbb{R}^{d}$ having bounded continuous partial
derivatives up to the second order in $t$ and third order in $x$,
respectively. Now we give the definition of viscosity solution for PIDE
(\ref{PIDE}).

\begin{definition}
A bounded upper semicontinuous (resp. lower semicontinuous) function $u$ on
$[0,1]\times \mathbb{R}^{d}$ is called a viscosity subsolution (resp. viscosity
supersolution) of (\ref{PIDE}) if $u(0,\cdot)\leq \phi(\cdot)$ $($resp.
$\geq \phi(\cdot))$ and for each $(t,x)\in(0,1]\times \mathbb{R}^{d}$,
\[
\partial_{t}\psi(t,x)-\sup \limits_{F_{\mu}\in \mathcal{L}}\left \{
\int_{\mathbb{R}^{d}}\delta_{\lambda}^{\alpha}\psi(t,x)F_{\mu}(d\lambda
)\right \}  \leq0\text{ }(\text{resp. }\geq0)
\]
whenever $\psi \in C_{b}^{2,3}((0,1]\times \mathbb{R}^{d})$ is such that
$\psi \geq u$ (resp. $\psi \leq u$) and $\psi(t,x)=u(t,x)$. A bounded continuous
function $u$ is a viscosity solution of (\ref{PIDE}) if it is both a viscosity
subsolution and supersolution.
\end{definition}

For each $\phi \in C_{b,Lip}(\mathbb{R}^{d})$, define
\begin{equation}
u(t,x)=\mathbb{\tilde{E}}[\phi(x+\tilde{\zeta}_{t})],\text{ }(t,x)\in
\lbrack0,1]\times \mathbb{R}^{d}. \label{u}%
\end{equation}

\begin{theorem}
\label{unique viscosity theorem}Suppose that the assumptions (A1)-(A2) hold.
Then, the value function $u$ of (\ref{u}) is the unique viscosity solution of
the fully nonlinear PIDE (\ref{PIDE}), i.e.,
\begin{equation}
\left \{
\begin{array}
[c]{l}%
\displaystyle \partial_{t}u(t,x)-\sup \limits_{F_{\mu}\in \mathcal{L}}\left \{
\int_{\mathbb{R}^{d}}\delta_{\lambda}^{\alpha}u(t,x)F_{\mu}(d\lambda)\right \}
=0,\\
\displaystyle u(0,x)=\phi(x),\text{\  \  \ }\forall(t,x)\in \lbrack
0,1]\times \mathbb{R}^{d},
\end{array}
\right.  \label{1.0}%
\end{equation}
where
\[
\delta_{\lambda}^{\alpha}u(t,x)=\left \{
\begin{array}
[c]{ll}%
u(t,x+\lambda)-u(t,x)-\langle D_{x}u(t,x),\lambda \mathbbm{1}_{\{|\lambda
|\leq1\}}\rangle, & \alpha=1,\\
u(t,x+\lambda)-u(t,x), & \alpha \in(0,1).
\end{array}
\right.
\]

\end{theorem}

\begin{proof}
We first show that $u$ is continuous. It is clear that $u(t,\cdot)$ is
uniformly Lipschitz continuous with the same Lipschitz constant as for $\phi$.
For each $t,s\in \lbrack0,1]$ such that $t+s\leq1$, we obtain
\begin{equation}
u(t+s,x)=\mathbb{\tilde{E}}[u(t,x+\tilde{\zeta}_{s})],\text{ \ }x\in
\mathbb{R}^{d}. \label{DPP}%
\end{equation}
Theorem \ref{The construction of Levy process} shows that there exists a
sequence $\{s_{k}\}_{k=1}^{\infty}\subset \mathcal{D}_{\infty}[0,1]$ satisfying
$s_{k}\downarrow s$ as $k\rightarrow \infty$ and a convergent sequence
$\big \{(s_{k}/n_{i}^{\ast})^{\frac{1}{\alpha}}S_{n_{i}^{\ast}}\big \}_{i=1}%
^{\infty}$ for each $s_{k}$, such that%
\begin{align}
\left \vert \mathbb{\tilde{E}}\big[u(t,x+\tilde{\zeta}_{s}%
)\big]-u(t,x)\right \vert  &  \leq L_{\phi}\mathbb{\tilde{E}}\big[|\tilde
{\zeta}_{s}|\wedge N_{\phi}\big]\nonumber \\
&  =L_{\phi}\lim \limits_{k\rightarrow \infty}\lim \limits_{i\rightarrow \infty
}\mathbb{\hat{E}}\big[|(s_{k}/n_{i}^{\ast})^{\frac{1}{\alpha}}S_{n_{i}^{\ast}%
}|\wedge N_{\phi}\big]\label{zata estimate}\\
&  \leq L_{\phi}s^{\frac{\delta}{\alpha}}\sup_{i}\mathbb{\hat{E}%
}\big[|(1/n_{i}^{\ast})^{\frac{1}{\alpha}}S_{n_{i}^{\ast}}|^{\delta
}\big]N_{\phi}^{1-\delta},\nonumber
\end{align}
where $L_{\phi}>0$ is the Lipschitz constant of $\phi$ and $N_{\phi}%
:=\frac{2\left \Vert \phi \right \Vert _{\infty}}{L_{\phi}}$. This implies the
continuity of $u(\cdot,x)$
\[
|u(t+s,x)-u(t,x)|\leq L_{\phi}M_{\delta}N_{\phi}^{1-\delta}s^{\frac{\delta
}{\alpha}}.
\]

Next, we will prove that $u$ is the unique viscosity solution of (\ref{1.0}).
The uniqueness of viscosity solution can be found in Corollary 55 in
\cite{HP2021}. It suffices to prove that $u$ is a viscosity subsolution, and
the other case can be proved in a similar way. Assume that $\psi$ is a smooth
test function on $(0,1]\times \mathbb{R}^{d}$ satisfying $\psi \geq u$ and
$\psi(\bar{t},\bar{x})=u(\bar{t},\bar{x})$ for some point $(\bar{t},\bar
{x})\in(0,1]\times \mathbb{R}^{d}$. For each $s\in(0,\bar{t})$, the dynamic
programming principle (\ref{DPP}) shows that
\begin{equation}
0=\mathbb{\tilde{E}}[u(\bar{t}-s,\bar{x}+\tilde{\zeta}_{s})-u(\bar{t},\bar
{x})]\leq \mathbb{\tilde{E}}[\psi(\bar{t}-s,\bar{x}+\tilde{\zeta}_{s}%
)-\psi(\bar{t},\bar{x})]. \label{1.1}%
\end{equation}
Using Taylor's expansion, we obtain that%
\begin{align}
&  \psi(\bar{t}-s,\bar{x}+\tilde{\zeta}_{s})-\psi(\bar{t},\bar{x})\nonumber \\
&  =\psi(\bar{t}-s,\bar{x}+\tilde{\zeta}_{s})-\psi(\bar{t},\bar{x}%
+\tilde{\zeta}_{s})+\psi(\bar{t},\bar{x}+\tilde{\zeta}_{s})-\psi(\bar{t}%
,\bar{x})\label{1.2}\\
&  =-\partial_{t}\psi(\bar{t},\bar{x})s+\psi(\bar{t},\bar{x}+\tilde{\zeta}%
_{s})-\psi(\bar{t},\bar{x})+\epsilon.\nonumber
\end{align}
where%
\[
\epsilon=s\int_{0}^{1}[-\partial_{t}\psi(\bar{t}-\theta s,\bar{x}+\tilde
{\zeta}_{s})+\partial_{t}\psi(\bar{t},\bar{x})]d\theta.
\]
Since $\psi$ is the smooth function, similar to (\ref{zata estimate}), one
easily gets $\mathbb{\tilde{E}}[|\epsilon|]\leq C_{\psi}s^{1+\frac{\delta
}{\alpha}}$. Then, it follows from (\ref{1.1})-(\ref{1.2})\ that
\begin{equation}
0\leq \lim_{s\rightarrow0}\mathbb{\tilde{E}}[\psi(\bar{t}-s,\bar{x}%
+\tilde{\zeta}_{s})-\psi(\bar{t},\bar{x})]s^{-1}=-\partial_{t}\psi(\bar
{t},\bar{x})+\lim_{s\rightarrow0}\mathbb{\tilde{E}}[\psi(\bar{t},\bar
{x}+\tilde{\zeta}_{s})-\psi(\bar{t},\bar{x})]s^{-1}. \label{1.3}%
\end{equation}
Moreover, Theorem \ref{represent theorem} implies that
\begin{equation}
\lim_{s\rightarrow0}\mathbb{\tilde{E}}\big[\psi(\bar{t},\bar{x}+\tilde{\zeta
}_{s})-\psi(\bar{t},\bar{x})\big]s^{-1}=\sup \limits_{F_{\mu}\in \mathcal{L}%
}\left \{  \int_{\mathbb{R}^{d}}\delta_{\lambda}^{\alpha}\psi(\bar{t},\bar
{x})F_{\mu}(d\lambda)\right \}  . \label{1.4}%
\end{equation}
Combining (\ref{1.3}) with (\ref{1.4}), we conclude that
\[
\partial_{t}\psi(\bar{t},\bar{x})-\sup \limits_{F_{\mu}\in \mathcal{L}}\left \{
\int_{\mathbb{R}^{d}}\delta_{\lambda}^{\alpha}\psi(\bar{t},\bar{x})F_{\mu
}(d\lambda)\right \}  \leq0.
\]
The proof is completed.
\end{proof}

\subsection{The scaling and symmetric properties}

\label{Section_Scaling and symmetric properties}

We start with the scaling property. In view of Theorem
\ref{The construction of Levy process}, for each $0\leq t\leq1$, there exists
a sequence $\{t_{k}\}_{k=1}^{\infty}\subset \mathcal{D}_{\infty}[0,1]$
satisfying $t_{k}\downarrow t$ as $k\rightarrow \infty$, such that for $\phi \in
C_{b,Lip}(\mathbb{R}^{d})$, $\mathbb{\tilde{E}}\big[\phi(\tilde{\zeta}_{t_{k}%
})\big]\rightarrow \mathbb{\tilde{E}}\big[\phi(\tilde{\zeta}_{t})\big]$ as
$k\rightarrow \infty$. For each fixed $t_{k}\in \mathcal{D}_{\infty}[0,1]$, we
assume that $t_{k}=l_{k}\tau_{m_{k}}=l_{k}2^{-m_{k}}$ for some $m_{k}%
\in \mathbb{N}$ and $0\leq l_{k}\leq2^{m_{k}}$. From the construction of
$\tilde{\zeta}_{t_{k}}$ and Remark \ref{ramark i.i.d.}, there exists a
convergent sequence $\{ \bar{S}_{n_{i}^{\ast}}^{\tau_{m_{k}}}\}_{i=1}^{\infty
}$ such that
\[
\mathbb{\tilde{E}}\big[\phi \big(\tilde{\zeta}_{t_{k}}\big)\big]=\mathbb{\tilde
{E}}^{m_{k}}\big[\phi \big(\tilde{\zeta}_{l_{k}\tau_{m_{k}}}%
\big)\big]=\mathbb{\hat{E}}_{1}\big[\phi \big(\zeta_{\tau_{m_{k}}}^{1}%
+\cdots+\zeta_{\tau_{m_{k}}}^{l_{k}}\big)\big]=\lim_{i\rightarrow \infty
}\mathbb{\hat{E}}\big[\phi \big(\bar{S}_{l_{k}n_{i}^{\ast}}^{t_{k}}\big)\big],
\]
for $\phi \in C_{b,Lip}(\mathbb{R}^{d})$, where
\[
\bar{S}_{l_{k}n_{i}^{\ast}}^{t_{k}}=(\tau_{m_{k}}/n_{i}^{\ast})^{1/\alpha
}(S_{l_{k}n_{i}^{\ast}}-S_{(l_{k}-1)n_{i}^{\ast}})+\cdots+(\tau_{m_{k}}%
/n_{i}^{\ast})^{1/\alpha}(S_{2n_{i}^{\ast}}-S_{n_{i}^{\ast}})+(\tau_{m_{k}%
}/n_{i}^{\ast})^{1/\alpha}S_{n_{i}^{\ast}}.
\]
In addition, from\ Theorem \ref{The construction of Levy process}, we obtain
that for $\phi(\sqrt[\alpha]{t_{k}}\cdot)\in C_{b,Lip}(\mathbb{R}^{d})$,
\[
\lim_{i\rightarrow \infty}\mathbb{\hat{E}}\big[\phi \big(\bar{S}_{l_{k}%
n_{i}^{\ast}}^{t_{k}}\big)\big]=\lim_{i\rightarrow \infty}\mathbb{\hat{E}%
}\big[\phi \big(\sqrt[\alpha]{t_{k}}\bar{S}_{l_{k}n_{i}^{\ast}}%
\big)\big]=\mathbb{\tilde{E}}\big[\phi \big(\sqrt[\alpha]{t_{k}}\tilde{\zeta
}_{1}\big)\big].
\]
This implies that
\[
\mathbb{\tilde{E}}\big[\phi \big(\tilde{\zeta}_{t}\big)\big]=\mathbb{\tilde{E}%
}\big[\phi \big(\sqrt[\alpha]{t}\tilde{\zeta}_{1}\big)\big],\text{ for }\phi \in
C_{b,Lip}(\mathbb{R}^{d}).
\]

Next, we shall verify that $\tilde{\zeta}_{\cdot}$ is symmetric. For any given
$\phi \in C_{b,Lip}(\mathbb{R}^{d})$, Theorem \ref{unique viscosity theorem}
implies that $u(t,0)=\mathbb{\tilde{E}}[\phi(\tilde{\zeta}_{t})]$, where $u$
is the unique viscosity solution of the PIDE (\ref{PIDE}) with initial
condition $\phi$. Note that
\[
F_{\mu}(B)=F_{\mu}(-B)\text{, \ for }B\in \mathcal{B}(\mathbb{R}^{d}).
\]
For any $0\leq t\leq1$, define $v(t,x):=u(t,-x)$. It follows from
\[
\int_{\mathbb{R}^{d}}\delta_{\lambda}^{\alpha}u(t,-x)F_{\mu}(d\lambda
)=\int_{\mathbb{R}^{d}}\delta_{\lambda}^{\alpha}v(t,x)F_{\mu}(d\lambda),
\]
that $v$ is the unique viscosity solution of the PIDE (\ref{PIDE}) with
initial condition $\psi(x):=\phi(-x)$. From Theorem
\ref{unique viscosity theorem}, we derive that $v(t,0)=\mathbb{\tilde{E}}%
[\psi(\tilde{\zeta}_{t})]$. Therefore, for any $\phi \in C_{b,Lip}%
(\mathbb{R}^{d})$,
\[
\mathbb{\tilde{E}}[\phi(\tilde{\zeta}_{t})]=u(t,0)=v(t,0)=\mathbb{\tilde{E}%
}[\psi(\tilde{\zeta}_{t})]=\mathbb{\tilde{E}}[\phi(-\tilde{\zeta}_{t})],
\]
and the proof is complete.

\appendix

\section{Appendix}

\subsection{Tightness under sublinear expectation}

The weak convergence approach plays an imperative role in establishing robust
limit theorems. We recall its main results for the reader's convenience.
Further details can be found in section 2 of \cite{HJLP2022}.

\begin{definition}
\label{def_2.7} A sublinear expectation $\mathbb{\hat{E}}$ on $(\mathbb{R}%
^{n},C_{b,Lip}(\mathbb{R}^{n}))$ is said to be tight if for each
$\varepsilon>0$, there exist an $N>0$ and $\varphi \in$ $C_{b,Lip}%
(\mathbb{R}^{n})$ with $\mathbbm{1}_{\{|x|\geq N\}}\leq \varphi$ such that
$\mathbb{\hat{E}}[\varphi]<\varepsilon$.
\end{definition}

\begin{definition}
\label{def_2.8} A family of sublinear expectations $\{ \mathbb{\hat{E}%
}_{\alpha}\}_{\alpha \in \mathcal{A}}$ on $(\mathbb{R}^{n},C_{b,Lip}%
(\mathbb{R}^{n}))$ is said to be tight if there exists a tight sublinear
expectation $\mathbb{\hat{E}}$ on $(\mathbb{R}^{n},C_{b,Lip}(\mathbb{R}^{n}))$
such that
\[
\mathbb{\hat{E}}_{\alpha}[\varphi]-\mathbb{\hat{E}}_{\alpha}[\varphi^{\prime
}]\leq \mathbb{\hat{E}}[\varphi-\varphi^{\prime}],\text{ for each }
\varphi,\varphi^{\prime}\in C_{b,Lip}(\mathbb{R}^{n}).
\]

\end{definition}

\begin{definition}
Let $\{ \mathbb{\hat{E}}_{n}\}_{n=1}^{\infty}$ be a sequence of sublinear
expectations defined on $(\mathbb{R}^{n},C_{b,Lip}(\mathbb{R}^{n}))$. They are
said to be weakly convergent if, for each $\varphi \in C_{b,Lip}(\mathbb{R}%
^{n})$, $\{ \mathbb{\hat{E}}_{n}[\varphi]\}_{n=1}^{\infty}$ is a Cauchy
sequence. A family of sublinear expectations $\{ \mathbb{\hat{E}}_{\alpha
}\}_{\alpha \in \mathcal{A}}$ defined on $(\mathbb{R}^{n},C_{b,Lip}%
(\mathbb{R}^{n}))$ is said to be weakly compact if for each sequence $\{
\mathbb{\hat{E}}_{\alpha_{i}}\}_{i=1}^{\infty}$ there exists a weakly
convergent subsequence.
\end{definition}

The following result is a generalization of the celebrated Prokhorov's theorem
to the sublinear expectation case.

\begin{theorem}
\label{tight theorem}Let $\{ \mathbb{\hat{E}}_{\alpha}\}_{\alpha \in
\mathcal{A}}$ be a family of tight sublinear expectations on $(\mathbb{R}%
^{n},C_{b,Lip}(\mathbb{R}^{n}))$. Then $\{ \mathbb{\hat{E}}_{\alpha}%
\}_{\alpha \in \mathcal{A}}$ is weakly compact, namely, for each sequence $\{
\mathbb{\hat{E}}_{\alpha_{n}}\}_{n=1}^{\infty}$, there exists a subsequence
$\{ \mathbb{\hat{E}}_{\alpha_{n_{i}}}\}_{i=1}^{\infty}$ such that, for each
$\varphi \in C_{b,Lip}(\mathbb{R}^{n})$, $\{ \mathbb{\hat{E}}_{\alpha_{n_{i}}%
}[\varphi]\}_{i=1}^{\infty}$ is a Cauchy sequence.
\end{theorem}

An immediate corollary of Theorem \ref{tight theorem} and Definition
\ref{converge in distribution} is the following result.

\begin{corollary}
\label{corollary tight theorem} {Let $\{X_{i}\}_{i=1}^{\infty}$ be a sequence
of $n$-dimensional random variables on a sublinear expectation space
$(\Omega,\mathcal{H},\mathbb{\hat{E})}$. If $\{ \mathbb{F}_{X_{i}}%
\}_{i=1}^{\infty}$ is tight, then there exists a subsequence $\{X_{i_{j}%
}\}_{j=1}^{\infty}\subset \{X_{i}\}_{i=1}^{\infty}$ which converges in
distribution.}
\end{corollary}

\subsection{{The regularity estimate for $u_{n}\label{Appendix 2}$}}

In this appendix, we shall state the regularity properties of $u_{n}$ defined
in (\ref{2.2}). For any fixed $N>0$, define $Z_{1}^{N}:=Z_{1}\mathbf{1}%
_{\{|Z_{1}|\leq N\}}$. It is easy to check that $\mathbb{\hat{E}}[Z_{1}%
^{N}]=\mathbb{\hat{E}}[-Z_{1}^{N}]=0$. We introduce the following truncated
scheme $u_{n,N}:[0,1]\times \mathbb{R\rightarrow R}$ recursively by
\begin{equation}%
\begin{array}
[c]{l}%
u_{n,N}(t,z)=|z|^{\delta},\text{ \ if\ }t\in \lbrack0,1/n),\\
u_{n,N}(t,z)=\mathbb{\hat{E}}[u_{n,N}(t-1/n,z+n^{-1/\alpha}Z_{1}^{N})],\text{
\ if\ }t\in \lbrack1/n,1].
\end{array}
\label{3.1}%
\end{equation}

It is easy to obtain the following estimates.

\begin{lemma}
\label{truncted_moment estimate}For any given $N>0$, we have%
\[%
\begin{array}
[c]{ll}%
\displaystyle \mathbb{\hat{E}}[|Z_{1}^{N}|^{2}]=N^{2-\alpha}I_{1,N},\text{ \ }
& \displaystyle \mathbb{\hat{E}}[|Z_{1}-Z_{1}^{N}|^{\delta}]=N^{\delta-\alpha
}I_{2,N},
\end{array}
\]
where
\begin{align*}
I_{1,N}  &  :=2\sup_{k\in K}\left \{  \frac{k}{2-\alpha}+2\int_{0}^{1}%
\frac{\beta_{k}(zN)}{z^{\alpha-1}}dz-\beta_{k}(N)\right \}  ,\\
I_{2,N}  &  :=2\sup_{k\in K}\left \{  \frac{k}{\alpha-\delta}+\delta \int
_{1}^{\infty}\frac{\beta_{k}(zN)}{z^{1+\alpha-\delta}}dz+\beta_{k}(N)\right \}
.
\end{align*}

\end{lemma}

\begin{lemma}
\label{u_N regularity}Given $N>0$. Then for any positive integer $k\leq n$ and
$x\in \mathbb{R}$,
\[
\left \vert u_{n,N}(k/n,x)-u_{n,N}(0,x)\right \vert \leq(I_{1,N})^{\frac{\delta
}{2}}N^{\frac{(2-\alpha)\delta}{2}}n^{\frac{(\alpha-2)\delta}{2\alpha}%
}(k/n)^{\frac{\delta}{2}},
\]

where $I_{1,N}$ is given in Lemma \ref{truncted_moment estimate}.
\end{lemma}

\begin{proof}
Applying induction to (\ref{3.1}), one can get that for any positive integer
$k\leq n$ and $x,y\in \mathbb{R}$
\[
\left \vert u_{n,N}(k/n,x)-u_{n,N}(k/n,y)\right \vert \leq|x-y|^{\delta}.
\]
From Young's inequality we know that for any $a,b>0$, $ab\leq \frac{\delta}%
{2}a^{\frac{2}{\delta}}+(1-\frac{\delta}{2})b^{\frac{2}{2-\delta}}$. For any
$\varepsilon>0$, let $x=|x-y|^{\delta}$ and $y=\frac{1}{\varepsilon}$, then it
follows from (i) that
\[
u_{n,N}(k/n,x)\leq u_{n,N}(k/n,y)+A|x-y|^{2}+B,
\]
where $A=\frac{\delta}{2}\varepsilon$ and $B=(1-\frac{\delta}{2}%
)\varepsilon^{\frac{-\delta}{2-\delta}}$.

We claim that, for any positive integer $k\leq n$ and $x,y\in \mathbb{R}$, it
holds that
\begin{equation}
u_{n,N}(k/n,x)\leq u_{n,N}(0,y)+A|x-y|^{2}+AM_{N}^{2}kn^{-\frac{2}{\alpha}}+B,
\label{3.3}%
\end{equation}
where $M_{N}^{2}:=\mathbb{\hat{E}}[|Z_{1}^{N}|^{2}]$. Indeed, (\ref{3.3})
obviously holds for $k=0$. Assume that for some integer $1\leq k<n$ the
assertion (\ref{3.3}) holds. Notice that
\begin{align}
u_{n,N}((k+1)/n,x)  &  =\mathbb{\hat{E}}[u_{n,N}(k/n,x+n^{-\frac{1}{\alpha}%
}Z_{1}^{N})]\label{3.4}\\
&  \leq u_{n,N}(0,y)+A\mathbb{\hat{E}}[|x-y+n^{-\frac{1}{\alpha}}Z_{1}%
^{N}|^{2}]+AM_{N}^{2}kn^{-\frac{2}{\alpha}}+B.\nonumber
\end{align}
Seeing that, $\mathbb{\hat{E}}[Z_{1}^{N}]=\mathbb{\hat{E}}[-Z_{1}^{N}]=0$, we
have
\begin{equation}
\mathbb{\hat{E}}[|x-y+n^{-\frac{1}{\alpha}}Z_{1}^{N}|^{2}]=|x-y|^{2}+M_{N}%
^{2}n^{-\frac{2}{\alpha}}. \label{3.5}%
\end{equation}
Combining (\ref{3.4})-(\ref{3.5}), we obtain that
\[
u_{n,N}((k+1)/n,x)\leq u_{n,N}(0,y)+A|x-y|^{2}+AM_{N}^{2}(k+1)n^{-\frac
{2}{\alpha}}+B,
\]
which shows that (\ref{3.3}) also holds for $k+1$. By the principle of
induction our claim is true for all positive integer $k\leq n$\ and
$x,y\in \mathbb{R}$. By taking $y=x$ in (\ref{3.3}), we have for any
$\varepsilon>0$,%
\[
u_{n,N}(k/n,x)\leq u_{n,N}(0,x)+\frac{\delta}{2}M_{N}^{2}kn^{-\frac{2}{\alpha
}}\varepsilon+(1-\frac{\delta}{2})\varepsilon^{\frac{-\delta}{2-\delta}}.
\]
By minimizing of the right-hand side with respect to $\varepsilon$, we obtain
that
\[
u_{n,N}(k/n,x)\leq u_{n,N}(0,x)+(M_{N}^{2})^{\frac{\delta}{2}}n^{\frac
{\delta(\alpha-2)}{2\alpha}}(k/n)^{\frac{\delta}{2}}.
\]
Similarly, we can also obtain the other side. From Lemma
\ref{truncted_moment estimate}, we obtain our desired result.
\end{proof}

Analogous to Lemma 4.5 in \cite{HJL2021}, we have the following error estimate
of $u_{n,N}$ for $u_{n}$.

\begin{lemma}
\label{num-truncated}Given $N>0$. Then for any positive integer $k\leq n$ and
$x\in \mathbb{R}$,
\[
\left \vert u_{n}(k/n,x)-u_{n,N}(k/n,x)\right \vert \leq I_{2,N}N^{\delta
-\alpha}n^{\frac{\alpha-\delta}{\alpha}}(k/n),
\]
where $I_{2,N}$ is given in Lemma \ref{truncted_moment estimate}.
\end{lemma}


Finally, we are in a position to obtain the regularity estimate of $u_{n}$.

\begin{theorem}
\label{u_num_regularity}For any $t,s\in \lbrack0,1]$ and $x\in \mathbb{R}$, we
have%
\[
\left \vert u_{n}(t,x)-u_{n}(s,x)\right \vert \leq I_{n}(|t-s|^{\delta
/2}+n^{-\delta/2}),
\]

where $I_{n}:=(I_{1,n^{1/\alpha}})^{\frac{\delta}{2}}+I_{2,n^{1/\alpha}}$ with
$I_{n}<\infty$ as $n\rightarrow \infty$.
\end{theorem}

\begin{proof}
Noting that $u_{n,N}(0,x)=u_{n}(0,x)=\phi(x)$, we have for any positive
integer $k\leq n$\
\[
\left \vert u_{n}(k/n,x)-u_{n}(0,x)\right \vert \leq \left \vert u_{n}%
(k/n,x)-u_{n,N}(k/n,x)\right \vert +\left \vert u_{n,N}(k/n,x)-u_{n,N}%
(0,x)\right \vert .
\]
In view of Lemmas \ref{u_N regularity} and \ref{num-truncated}, by choosing
$N=n^{\frac{1}{\alpha}}$, we obtain
\begin{equation}%
\begin{split}
\left \vert u_{n}(k/n,x)-u_{n}(0,x)\right \vert  &  \leq(I_{1,N})^{\frac{\delta
}{2}}N^{\frac{(2-\alpha)\delta}{2}}n^{\frac{(\alpha-2)\delta}{2\alpha}%
}(k/n)^{\frac{\delta}{2}}+I_{2,N}N^{\delta-\alpha}n^{\frac{\alpha-\delta
}{\alpha}}(k/n)\\
&  \leq \big((I_{1,n^{1/\alpha}})^{\frac{\delta}{2}}+I_{2,n^{1/\alpha}%
}\big)(k/n)^{\frac{\delta}{2}}.
\end{split}
\label{3.8}%
\end{equation}
Moreover, it is easy to verify that $I_{1,n^{1/\alpha}}$ and $I_{2,n^{1/\alpha
}}$ are finite as $n\rightarrow \infty$. Without loss of generality, we assume
$k\geq l$. By using induction (\ref{2.2}) and the estimate (\ref{3.8}), we
obtain that for any $k\geq l$ and $x\in \mathbb{R}$,
\begin{equation}%
\begin{split}
&  \left \vert u_{n}(k/n,x)-u_{n}(l/n,x)\right \vert \\
&  =\big \vert \mathbb{\hat{E}}\big [u_{n}\big ((k-l)/n,x+n^{-\frac{1}{\alpha
}}\sum_{i=1}^{l}Z_{i}\big )\big ]-\mathbb{\hat{E}}\big [u_{n}%
\big (0,x+n^{-\frac{1}{\alpha}}\sum_{i=1}^{l}Z_{i}\big )\big ]\big \vert \\
&  \leq \big((I_{1,n^{1/\alpha}})^{\frac{\delta}{2}}+I_{2,n^{1/\alpha}%
}\big)((k-l)/n)^{\frac{\delta}{2}}.
\end{split}
\label{3.9}%
\end{equation}
In general, for $s,t\in \lbrack0,1]$, let $\beta_{s},\beta_{t}\in \lbrack0,1/n)$
such that $s-\beta_{s}$ and $t-\beta_{t}$ are in the grid points
$\{k/n:k\in \mathbb{N}\}$. Then, from (\ref{3.9}), we have%
\begin{align*}
u_{n}(t,x)=u_{n}(t-\beta_{t},x)  &  \leq u_{n}(s-\beta_{s}%
,x)+\big((I_{1,n^{1/\alpha}})^{\frac{\delta}{2}}+I_{2,n^{1/\alpha}%
}\big)|t-s-\beta_{t}+\beta_{s}|^{\frac{\delta}{2}}\\
&  \leq u_{n}(s,x)+\big((I_{1,n^{1/\alpha}})^{\frac{\delta}{2}}%
+I_{2,n^{1/\alpha}}\big)(|t-s|^{\frac{\delta}{2}}+n^{-\frac{\delta}{2}}).
\end{align*}
Similarly, we can prove the other side. The proof is complete.
\end{proof}



\end{document}